\documentclass[]{journal}
\usepackage[mathscr]{eucal}
\usepackage[latin1]{inputenc}
\usepackage{amsmath}
\usepackage{amsthm}
\usepackage{authblk}
\usepackage{amssymb}
\usepackage{amscd}
\usepackage{natbib}
\usepackage{mathrsfs}
\usepackage[pdftex]{graphicx}
\usepackage{graphicx}
\usepackage[OT2, T1]{fontenc}
\usepackage[russian,english]{babel}
\usepackage[babel=true]{csquotes}
\usepackage{chngcntr}
\counterwithout{figure}{section}
\counterwithout{figure}{subsection}
\newcommand{\ue}[0]{\ensuremath{\mathbf{u}}}
\newcommand{\ve}[0]{\ensuremath{\mathbf{v}}}

\newcommand{\ix}[0]{\ensuremath{\mathbf{x}}}
\newcommand{\y}[0]{\ensuremath{\mathbf{y}}}
\newcommand{\ka}[0]{\ensuremath{\mathbf{k}}}
\newcommand{\en}[0]{\ensuremath{\mathbf{n}}}
\newcommand{\er}[0]{\ensuremath{\mathbf{r}}}
\newcommand{\zero}[0]{\ensuremath{\mathbf{0}}}

\newcommand{\e}[0]{\ensuremath{\mathbf{e}}}
\newcommand{\be}[0]{\ensuremath{\mathbf{b}}}
\newcommand{\E}[0]{\ensuremath{\mathbf{E}}}
\newcommand{\F}[0]{\ensuremath{\mathcal{F}}}

\newcommand{\Pe}[0]{\ensuremath{\mathbf{P}}}
\newcommand{\R}[0]{\ensuremath{\mathbb{R}}}
\newcommand{\T}[0]{\ensuremath{\mathbb{T}}}

\newcommand{\Z}[0]{\ensuremath{\mathbb{Z}}}

\newtheorem{theo}{Theorem}[section]{\bf}{\it}
\newtheorem{lemm}[theo]{Lemma}{\bf}{\it}
\newtheorem{defi}[theo]{Definition}{\bf}{\it}
\newtheorem{cor}[theo]{Corollary}{\bf}{\it}
\newtheorem{rmq}[theo]{Remark}{\bf}{\it}

\author{Alexandre Boritchev}
\affil{
University of Lyon
\\
CNRS UMR 5208
\\
University Claude Bernard Lyon 1
\\
Institut Camille Jordan
\\
43 Blvd. du 11 novembre 1918
\\
69622 VILLEURBANNE CEDEX
\\
FRANCE
\\
E-mail: alexandre.boritchev@gmail.com
}
\title{Multidimensional potential Burgers turbulence}
\date{\today}
\begin{document}


\maketitle

\bigskip
\textbf{Abstract.}\ We consider the multidimensional generalised stochastic Burgers equation in the space-periodic setting:
\begin{equation} \nonumber
\frac{\partial \ue}{\partial t}+(\nabla f(\ue) \cdot \nabla) \ue -\nu \Delta \ue= \nabla \eta,\quad t \geq 0,\ \ix \in \T^d=(\R/\Z)^d,
\end{equation}
under the assumption that $\ue$ is a gradient. Here $f$ is strongly convex and satisfies a growth condition, $\nu$ is small and positive, while $\eta$ is a random forcing term, smooth in space and white in time.
\\ \indent
For solutions $\ue$ of this equation, we study Sobolev norms of $\ue$ averaged in time and in ensemble: each of these norms behaves as a given negative power of $\nu$. These results yield sharp upper and lower bounds for natural analogues of quantities characterising the hydrodynamical turbulence, namely the averages of the increments and of the energy spectrum. These quantities behave as a power of the norm of the relevant parameter, which is respectively the separation $\ell$ in the physical space and the wavenumber $\ka$ in the Fourier space. Our bounds do not depend on the initial condition and hold uniformly in $\nu$. 
\\ \indent
We generalise the results obtained for the one-dimensional case in \cite{BorW}, confirming the physical predictions in \cite{BK07,GMN10}. Note that the form of the estimates does not depend on the dimension: the powers of $\nu, |\ka|, \ell$ are the same in the one- and the multi-dimensional setting.

\tableofcontents

\listoffigures

\section*{Abbreviations}

\begin{itemize}
\item 1d, 2d, multi-d: 1, 2, multi-dimensional
\item a.e.: almost every
\item a.s.: almost surely
\item (GN): the Gagliardo--Nirenberg inequality (Lemma~\ref{GN})
\item i.i.d.: independent identically distributed
\item r.v.: random variable
\end{itemize}

\section{Introduction}

\subsection{Burgers turbulence}

\indent
The multi-dimensional generalised Burgers equation
\begin{equation} \label{Burbegin}
\frac{\partial \ue}{\partial t} +  (\nabla f(\ue) \cdot \nabla) \ue  - \nu \Delta \ue = \zero,\ \ix \in \R^d,\ \ue(t,\ix) \in \R^d,
\end{equation}
where $\nu>0$ is a constant (the classical Burgers equation \cite{Bur74} corresponds to $f(\ue)=|\ue|^2/2$) is historically a popular model for the Navier-Stokes equations, since both of them have similar nonlinearities and dissipative terms. 
\\ \indent
Taking the curl of (\ref{Burbegin}), we see that for a gradient initial condition $\ue^0=\nabla \psi^0$, the solution $\ue$ remains a gradient for all times. Namely, this solution is the gradient of the solution $\psi(t,\cdot)$ to the viscous generalised Hamilton--Jacobi equation
\begin{equation} \label{HJbegin}
\frac{\partial \psi}{\partial t} +  f(\nabla \psi)  - \nu \Delta \psi = 0
\end{equation}
with the initial condition $\psi^0$. For shortness, in this case we will write the Burgers equation as
\begin{equation} \label{Burbeginfirst}
\frac{\partial \ue}{\partial t} +  (\nabla f(\ue) \cdot \nabla) \ue  - \nu \Delta \ue = \zero;\ \ue=\nabla \psi,\ \nu>0,
\end{equation}
where it is implicitly assumed that the potential $\psi$ satisfies (\ref{HJbegin}). We will do likewise for the equation (\ref{Burbeginfirst}) with a gradient right part instead of $\zero$, and we will say that we are in the \textit{potential case}. From now on, unless otherwise stated, we will only consider this case. Moreover, we will only consider the space-periodic setting, i.e.
$$
\ix \in \T^d=(\R/\Z)^d.
$$
The mathematical advantage of the potential case is that the equation (\ref{HJbegin}) can be treated by variational methods (see for instance \cite{GIKP05}). Moreover, for $f(\ue)=|\ue|^2/2$ the equation (\ref{Burbeginfirst}) has become popular as a model in astrophysics: in the limit $\nu \rightarrow 0$, it corresponds to the adhesion approximation introduced by Gurbatov and Saichev and developed later by Shandarin and Zeldovich \cite{GurSai84,GMN10,ShaZel89}. The equation (\ref{Burbegin}) is also relevant for fields as different as statistical physics, geology and traffic modelling (see the surveys \cite{BF01,BK07} and references therein; see also \cite{Flo48}). 
\\ \indent
For $f(\ue)=|\ue|^2/2$, the equation (\ref{Burbeginfirst}) can be transformed into the heat equation by the Cole-Hopf method \cite{Col51, Hop50}. In some settings (for instance when considering the Burgers equation with very singular additive noise) this method is extremely helpful (see \cite{BK07} and references therein). However, it is harder to make use of this transformation in the setting of our paper, where we are concerned with the quantitative behaviour of solutions in the singular limit $\nu \rightarrow 0^+$. Moreover, the Cole-Hopf method does not allow us to study the Burgers equation for a nonlinearity other than $f(\ue)=|\ue|^2/2$.
\\ \indent
When studying the \textit{local} (in space) fine structure of a function, natural objects of interest are the \textit{small-scale quantities}, which play an important role in the study of turbulence \cite{Fri95}. In the physical space, this denomination includes the structure functions (i.e. the moments of increments in space) for small separations. In the Fourier space, an important quantity of interest is the energy spectrum on small scales (i.e. the amount of energy carried by high Fourier modes). It is important to understand the critical thresholds for the relevant parameters (respectively, in the physical space the separation distance and in the Fourier space the wavenumber) between regions where the small-scale quantities exhibit different types of behaviour. These values are referred to as \textit{length scales}.
\\ \indent
The systematic study of small-scale quantities for the solutions of nonlinear PDEs with a small parameter with or without random forcing was initiated by Kuksin. He obtained lower and upper estimates of these quantities by negative powers of the parameter for a large class of equations (see \cite{Kuk97GAFA,Kuk99GAFA} and the references in \cite{Kuk99GAFA}). A natural way to study these quantities is through upper and lower bounds for Sobolev norms: for a discussion of the relationship between Sobolev norms and spatial scales, see \cite{Kuk99GAFA}.  For more recent results obtained for the 2D Navier-Stokes equations, see the monograph \cite{KuSh12} and the references therein. 
\\ \indent
Before treating the multi-d case, we recall some facts about the behaviour of the solutions to (\ref{Burbegin}) in the 1d setting. We only consider the case where $f$ is strongly convex, i.e. there exists $\sigma>0$ such that
\begin{equation} \label{1dconvex}
f''(x) \geq \sigma,\ x \in \R.
\end{equation}
In this setting, the requirement that we are in the potential case implies the vanishing of the space average of the solution.
\\ \indent
We consider the regime $\nu \ll 1$. Since all other parameters are fixed, in the hydrodynamical language this corresponds to the case of a large Reynolds number. Under these assumptions, the solutions display turbulent-like behaviour, called Burgers turbulence or \enquote{Burgulence} \cite{Bur74,Cho75,Kid79}, which we describe now.
\\ \indent
 In the limit $\nu \rightarrow 0$ and for large enough times, we observe $N$-waves, i.e. the graphs of the solutions $u(t,\cdot)$ are composed of waves similar to the Cyrillic capital letter \textit{\foreignlanguage{russian}{I}} (the mirror image of $N$). In other words, at a time $t_0$ the solution stops being smooth, and for times $t>t_0$ the solution $u(t,\cdot)$ alternates between negative jump discontinuities and smooth regions where the derivative is positive and of the order $1/t$ (see for instance \cite{Eva08}). Thus, it exhibits \textit{small-scale spatial intermittency} \cite{Fri95}, i.e. for a fixed time the excited behaviour only takes place in a small region of space. For $0<\nu \ll 1$ the solutions are still highly intermittent: shocks become zones where the derivative is small and positive, called {\it ramps}, which alternate with zones where the derivative is large in absolute value and negative, called {\it cliffs} (cf. Figure~\ref{N}).
\\ \indent
\begin{figure} 
\centering
\includegraphics[width=9cm, height=5cm]{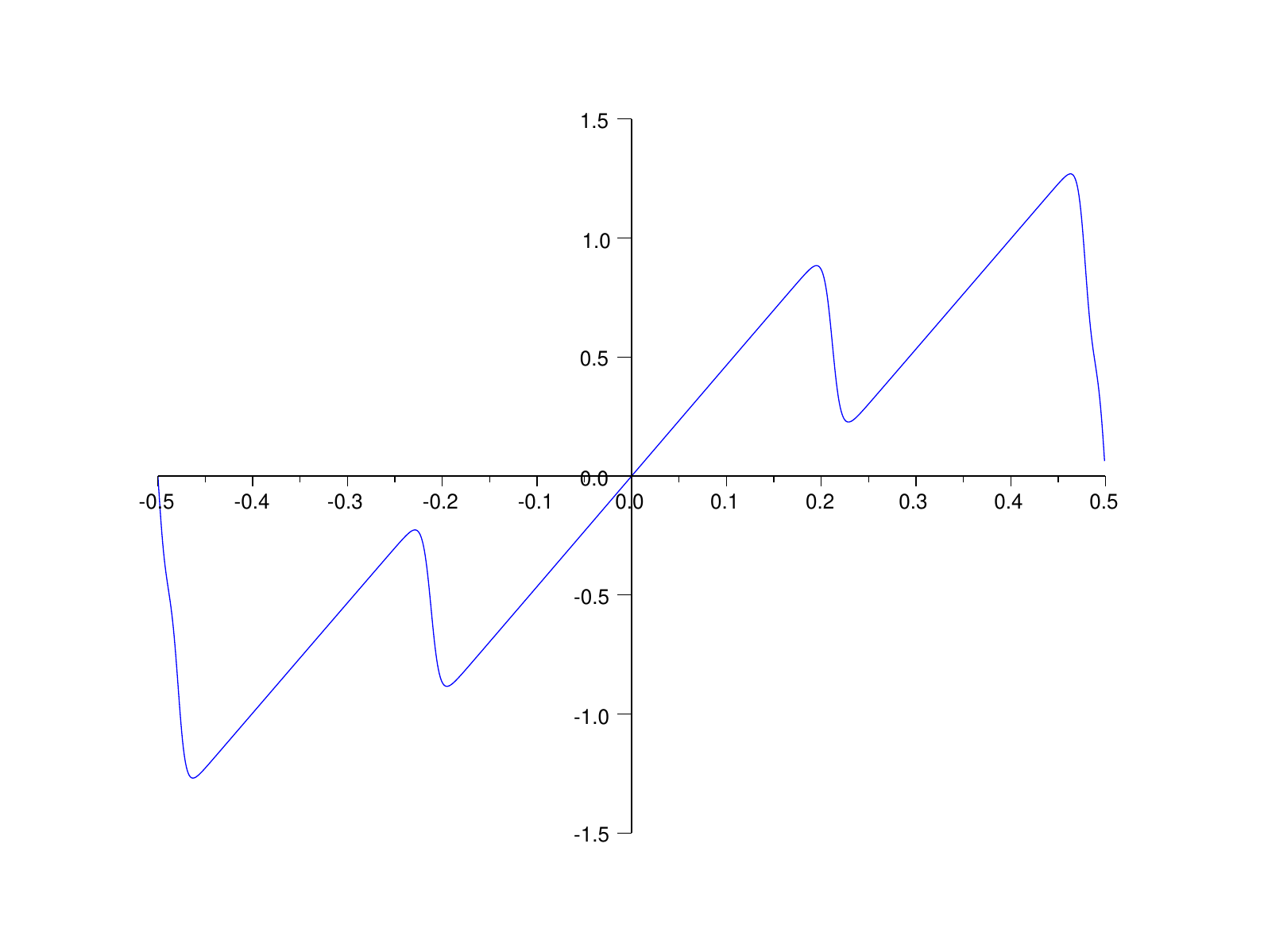}
\caption{\protect\label{N}  \enquote{Typical} solution of the 1d Burgers equation}
\end{figure}
For the prototypical $N$-wave, i.e. for the $1$-periodic function equal to $x$ on $(-1/2,\ 1/2]$, the Fourier coefficients satisfy $|\hat{u}(k)| \sim k^{-1}$. On the other hand, for $0<\nu \ll 1$ the dissipation gives exponential decay of the spectrum for large values of $k$. This justifies the conjecture that for $\nu$ small and for "moderately large" values of $k$, the energy-type quantities $\frac{1}{2} |\hat{u}(k)|^2$ behave, in average, as $k^{-2}$ \cite{Cho75,FouFri83,Kid79,Kra68}.
\\ \indent
In the physical space, the natural analogues of the energy $\frac{1}{2} |\hat{u}(k)|^2$ at the wavelength $k^{-1}$ are the structure functions
\begin{equation} \label{Sp1d}
S_p(\ell)=\int_{S^1}{|u(x+\ell)-u(x)|^p\ dx}.
\end{equation}
Heuristically, the behaviour of the solutions which is described above implies that for $\nu \ll \ell \ll 1$, these quantities behave as $\ell^{\min(1,p)}$ for $p \geq 0$:  see \cite{AFLV92} and the introduction to \cite{BorD}. 
\\ \indent
\begin{figure}
\centering
\includegraphics[width=16cm, height=12cm]{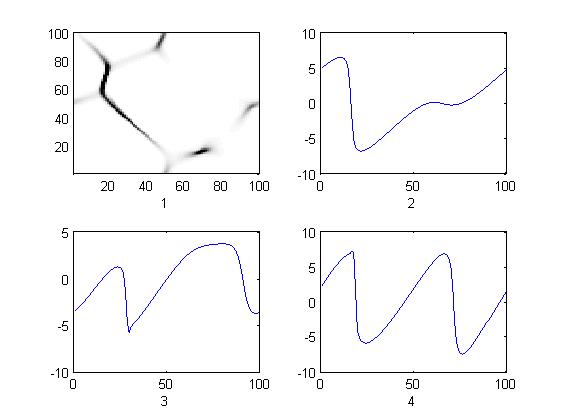}
\caption[\enquote{Typical} solution of the 2d potential Burgers equation]{\protect\label{N2d}  \enquote{Typical} solution of the 2d potential Burgers equation.
\\
1: The value of $\partial^2 \psi/\partial x_1^2 + \partial^2 \psi/\partial x_2^2$. The shaded regions correspond to zones where this value is large, which in the limit $\nu \rightarrow 0$ correspond to the shock manifold.
\\
2, 3, 4: 1d projections of the multi-d solution, respectively along the horizontal axis $e_1$, the vertical axis $e_2$ and the diagonal axis $e_1-e_2$.}
\end{figure}
Now we consider the potential multi-d case. In the case $f(\ue)=|\ue|^2/2$, in the inviscid limit $\nu \rightarrow 0$ it is numerically observed that the behaviour of the solution is analogous to what is happening in 1d \cite{BK07}. Namely, for large enough times one observes a tesselation where cells inside which solutions are smooth are separated by 1-codimension shock manifolds. In particular, in average, 1d projections of the multi-d solution look like the 1d solution (cf. Figure~\ref{N2d}).
\\ \indent
Thus, it is reasonable to expect a behaviour of the longitudinal structure functions
\begin{align} \label{Spintro}
S^{\|}_{p}(\er)=&  \int_{\ix \in \T^d }{\Bigg| \frac{(\ue(\ix+\er)-\ue(\ix)) \cdot \er}{|\er|} \Bigg|^p d \ix}
\end{align}
for a certain range of values of $|\er|$ which is analogous to the behaviour of the structure functions in 1d (at least after averaging with respect to $\er$ for a fixed value of $|\er|$). Similarly, we could expect spectral asymptotics of the type 
$$
 \frac{\sum_{| \en | \sim k}{|\hat{\ue}(\en)|^2}}{\sum_{| \en | \sim k}{1}} \sim k^{-2},
$$
for a certain range of values of $k$.
\\ \indent
By analogy with the 1d case, one can conjecture that we have the same behaviour for $f$ strongly convex, which in multi-d means that
\begin{equation} \label{strconvex}
\frac{^{t} \ve D^2f(x) \ve}{| \ve |^2} \geq \sigma > 0,\quad \ve=(v_1,\dots,v_d) \neq \zero,\quad x \in \R^d,
\end{equation}
where $D^2$ is the Hessian matrix and $|\ve|$ is the norm
$$
\sqrt{v_1^2+\dots+v_d^2}.
$$
\\ \indent
Now let us say a few words about the similarities and the differences between the multi-dimensional potential Burgulence and the real incompressible turbulence. It is clear that the geometric pictures on small scales are quite different for these two models: the multi-dimensional analogues of $N$-waves created by infinitely strong compressibility do not have the same nature as the complex multi-scale structures modeled by incompressibility constraints such as the vortex tubes. However, the similarity in the form of the potential Burgers equation and the incompressible Navier-Stokes equations implies that some physical arguments justifying different theories of turbulence can be applied to the Burgulence. Indeed, both models exhibit an inertial nonlinearity of the form $u \cdot \nabla u$, and a viscous term which in the limit $\nu \rightarrow 0$ gives a dissipative anomaly \cite{Fri95}. Hence, the Burgers equation is often used as a benchmark for turbulence theories, as well as for numerical methods for the Navier-Stokes equations. For more information on both subjects, see \cite{BK07}.

\subsection{State of the art and setting} \label{state}

For the \textit{unforced} Burgers equation, some \textit{upper} estimates for Sobolev norms of solutions and for small-scale quantities are well-known. For references on classical aspects of the theory of scalar (viscous or inviscid) conservation laws, see \cite{Daf10,Lax06,Ser99}. For some upper estimates for small-scale quantities, see \cite{Kre88,Tad93}. To our best knowledge, rigorous lower estimates were not known before Biryuk's and our work.
\\ \indent
In \cite{Bir01}, Biryuk considered the unforced generalised Burgers equation (\ref{Burbegin}) in the 1d space-periodic case, with $f$ satisfying (\ref{1dconvex}). He obtained estimates for $L_2$ Sobolev norms of the $m$-th spatial derivatives of the solutions:
\begin{equation} \nonumber
\frac{1}{T} \int_{0}^{T}{\Vert u(t)\Vert^2_m} \sim \nu^{-(2m-1)},\quad m \geq 1,\ \nu \leq \nu_0.
\end{equation}
The constants $\nu_0$ and $T$ and the multiplicative constants implicitly contained in the symbol $\sim$ depend on the deterministic initial condition $u^0$ as well as on $m$. Biryuk also obtained almost sharp spectral estimates which allowed him to give the correct value of the dissipation scale, which equals $\nu$ (see Section~\ref{agreeturb} for its definition). We can explain Biryuk's method by a dimensional analysis argument, considering the quantity
$$
A_m= \frac{\Vert u(t)\Vert_m}{\Vert u(t)\Vert_{m+1}}
$$
(see \cite{Kuk97GAFA,Kuk99GAFA}). Indeed, after averaging in time one gets
$$
A_m \sim \nu,\quad m \geq 1,
$$
as $\nu \rightarrow 0$.
\\ \indent
In \cite{BorD}, we generalised Biryuk's estimates to the $L_p$ Lebesgue norms of the $m$-th spatial derivatives for $1 <p \leq \infty$. Moreover, we  improved Biryuk's estimates for small-scale quantities, obtaining sharp $\nu$-independent estimates. In particular, for $\ell \in [C \nu,C]$, we proved that
$$
\frac{1}{T_2-T_1} \int_{T_1}^{T_2}{S_{p}(\ell) dt} \sim \left\lbrace \begin{aligned} & \ell^{p},\ 0 \leq p \leq 1. \\ & \ell,\ p \geq 1, \end{aligned} \right.
$$
with $S_p(\ell)$ defined by (\ref{Sp1d}), and for $k$ such that $k^{-1} \in [C \nu,C]$, we obtained that
$$
\frac{1}{T_2-T_1} \int_{T_1}^{T_2}{ \frac{\sum_{n \in [M^{-1}k,Mk]}{|\hat{u}^n|^2}}{\sum_{|n| \in [M^{-1}k,Mk]}{1}} dt} \sim k^{-2}.
$$
The constants $\nu_0$, $M$, $T_1$ and $T_2$, as well as the different strictly positive constants denoted by $C$ and the multiplicative constants implicitly contained in the symbol $\sim$ depend on the deterministic initial condition $u^0$ as well as on $p$. Note that here again, the upper and the lower estimates only differ by a multiplicative constant. Moreover, we rigorously prove that $k^{-1} \sim \nu$ is the threshold parameter which corresponds to the transition between algebraic (in $k^{-2}$) and super-algebraic behaviour of the energy spectrum.
\\ \indent
To get results independent of the initial data, a natural idea is to introduce random forcing and to average with respect to the corresponding probability measure. In the articles \cite{BorK,BorW}, we have considered the 1d case with $\zero$
in the right-hand side of (\ref{Burbegin}) replaced by a random spatially smooth force, \enquote{kicked} and white in time, respectively. In the \enquote{kicked} model, we consider the unforced equation and at integer times, we add i.i.d. smooth in space impulsions. The white force corresponds, heuristically, to a scaled limit of \enquote{kicked} forces with more and more frequent kicks. On a formal level, this can be explained by Donsker's theorem, since by definition a white force is the weak derivative in time of a Wiener process.
\\ \indent
In the random case, the estimates for the Sobolev norms and for the small-scale quantities seem at first sight to be almost word-to-word the same as in the unforced case. However, there are two major differences. The first one is that along with the averaging in time we also need to take the expected value. The second one is that we have estimates which hold uniformly with respect to the starting time $T_1$ for intervals $[T_1,T_1+T]$ of fixed length on which we consider the averaged quantities; moreover, the constants in the bounds do not any more depend on the initial condition.
\\ \indent
To explain the second difference, we observe that in the unforced case, no energy source is available to counterbalance the viscous dissipation, whereas in the forced case the stochastic term provides such a source. Thus, the existence of a stationary measure which is nontrivial (i.e., not proportional to the Dirac measure $\delta_{0}$) is possible in the randomly forced case, as opposed to the unforced case where we have a decay to $0$ of the solutions at the speed $Ct^{-1}$. In the language of statistical physics, this corresponds to the existence of a non-trivial non-equilibrium steady state \cite{Gal02}. Indeed, in \cite{BorW} we prove the existence and the uniqueness of the stationary measure for the generalised white-forced Burgers equation; our arguments also apply to the kick-forced case. For more details on Biryuk's and our work on 1d Burgulence, see the survey \cite{BorS}.
\\ \indent
In this paper,  we study the white-forced equation
\begin{equation} \label{whiteintro}
\frac{\partial \ue}{\partial t} +  (\nabla f(\ue) \cdot \nabla) \ue  - \nu \Delta \ue = \nabla \eta;\ u=\nabla \psi,\ \nu>0,\ \ix \in \T^d,
\end{equation}
under the additional convexity and growth assumptions  (\ref{strconvex}, \ref{poly}) on $f$. We obtain estimates for the Sobolev norms and the small-scale quantities which are (up to some changes in definitions due to the multi-dimensional setting) word-to-word the same as those proved in \cite{Bir01,BorK,BorW,BorD}, with the same exponents for $\nu$. The only small difference between the results in 1d and in this article is that we do not obtain upper estimates for the $W^{m,\infty}$ norms. Moreover, we obtain results on the existence and the uniqueness of the stationary measure $\mu$ for the equation (\ref{whiteintro}) as well as the rate of convergence to $\mu$. Thus, we generalise the 1d results in \cite{BorW}.
\\ \indent
The assumption that $\ue$ is a gradient plays a crucial role, since it allows us to generalise the 1d arguments from the papers \cite{Bir01,BorW}, in particular for the upper estimates; see Theorem~\ref{uxpos}. On the other hand there is a major difficulty specific to the multi-d case. Namely, the energy balance is much more delicate to deal with than in 1d; see Section~\ref{lower}. This is the reason why here, unlike in 1d, we assume that the noise $\eta$ is \enquote{diagonal}: in other words, there is no correlation between the different Fourier modes. This allows us to use a more involved version of the "small-noise zones" argument (see for instance \cite{IK03}). Roughly speaking, this argument tells that if the noise is small during a long time interval, then the solution of the generalised Burgers equation goes to $0$, roughly at the same rate as if there was no noise at all, i.e. at least as $Ct^{-1}$. Note that by classical properties of Wiener processes, such an interval will eventually occur with probability $1$: see \cite[Formula (10)]{BorW} for a quantitative version of this statement.
\\ \indent
In \cite{Bir04}, Biryuk studied solutions of the space-periodic multi-d Burgers equation without the assumption that $\ue$ is a gradient. He obtained upper and lower estimates which are non-sharp, in the sense that there is a gap between the powers of $\nu$ for the upper and the lower estimates. In a setting very similar to ours, Brzezniak, Goldys and Neklyudov \cite{BGN14,GN09} have considered the multi-d Burgers equation both in the deterministic and  in the stochastic case, obtaining results on the well-posedness both in the whole-space and in the periodic setting. Moreover, in the potential space-periodic case those authors have obtained estimates which are uniform with respect to the viscosity coefficient $\nu$; however, those estimates are not uniform in time, unlike the ones proved in our paper.
\\ \indent
We are concerned with solutions for small but positive $\nu$. For a study of the limiting dynamics with $\nu=0$, see \cite{EKMS97,EKMS00} for the 1d case, \cite{GIKP05,IK03} for the multi-d case, and \cite{DV,DS05} for the case of multi-d \textit{scalar} conservation laws with nonconvex flux.
\\ \indent
In \cite{Bir01,BorK,BorW,BorD} as well as in our paper, estimates on Sobolev norms and on small-scale quantities are asymptotically sharp in the sense that $\nu,\ell,k$ enter lower and upper bounds at the same power. Such estimates are not available for the more complicated equations considered in \cite{Kuk97GAFA,Kuk99GAFA,KuSh12}. Another remarkable feature of our estimates is that the powers of the quantities $\nu,\ell,k$ are always the same as in 1d. Thus, those estimates are in agreement with the physical predictions for space increments \cite[Section 7]{BK07} and for spectral asymptotics \cite{GMN10} of the solutions $\ue(t,x)$.
\\ \indent
The results of our paper extend to the case of a \enquote{kicked} force, under some restrictions. Namely, while the upper estimates hold in a very general setting, to prove the lower estimates we seem to need some non-trivial assumptions on the support of the kick, since the dissipation relation for the energy $1/2 \int_{\T^d}{|\ue|^2}$ has an additional trilinear term compared to the 1d case. For the same reason, the results in the unforced case are expected to be less general than in 1d.
\\ \indent
To prove our results on the existence and the uniqueness of the stationary measure and the rate of convergence to it, we use a quantitative version of the "small-noise zones" argument \cite{IK03}, a coupling argument due to Kuksin and Shirikyan \cite{KuSh12} and $L_{\infty}$-contractivity for the flow of the Hamilton-Jacobi equation satisfied by the potential $\psi$.

\subsection{Plan of the paper} \label{plan}

After introducing the notation and the setup in Section~\ref{nota}, we formulate the main results in Section~\ref{results}. In Section~\ref{upper}, for $t \geq 1$ and for a vector $\ka \in \R^d$ with integer coefficients, we begin by estimating from above the moments of the quantities
$$
\max_{s \in [t,t+1],\ \ix \in \T^d}{(\ka \cdot \nabla)^2 \psi(s,\ix)}
$$
for the potential $\psi$ corresponding to the solution $\ue(t,\ix)$ of (\ref{whiteintro}). 
\\ \indent
In Sections~\ref{upper}-\ref{main} we get estimates for the Sobolev norms of the same type as those obtained in \cite{Bir01,BorK,BorW,BorD} with the same exponents for $\nu$, valid for time $t \geq T_0$; the only small difference with the 1d case is that here we do not obtain sharp upper bounds for the $W^{m,\infty}$ norms. Here, $T_0$ is a constant, independent of the initial condition and of $\nu$. Actually, for $t \geq T_0$, we are in a \textit{quasi-stationary regime}: all the estimates hold uniformly in $t$ and in the initial condition $\ue^0$.
\\ \indent
In Section~\ref{turb} we study the implications of our results in terms of the theory of Burgulence. Namely, we give sharp upper and lower bounds for the dissipation length scale, the increments and the spectral asymptotics for the flow $\ue(t,x)$. These bounds hold uniformly for $\nu \leq \nu_0$, where $\nu_0$ is a constant which is independent of the initial condition. One proof in this section uses (indirectly) a 1d argument from \cite{AFLV92}.
\\ \indent
In Section~\ref{stat}, we prove the existence and the uniqueness of the stationary measure for the equation (\ref{whiteintro}), and we give an estimate for the speed of convergence to this stationary measure.

\section{Notation and setup} \label{nota}

\subsection{Functions, indices, derivatives}
All functions that we consider are real-valued or, if written in bold script, vector-valued. When giving formulas which hold for functions which can be scalar or vector-valued, we use the usual script. We denote by $(\e_1,\dots,\e_d)$ the canonical vector basis of $\R^d$. We assume that $d \geq 2$. Note that all of our estimates still hold for $d=1$: see \cite{BorW}.
\\ \indent
The subscript $t$ denotes partial differentiation with respect to the variable $t$. When we consider a scalar-valued function $v$, the subscripts $i,\ 1 \leq i \leq N$, which can be repeated, denote differentiation with respect to the variables $x_i,\ 1 \leq i \leq N$, respectively. Since the only scalar-valued functions $v$ for which the notation $v_{i_1,\dots,i_k}$ will be used are infinitely differentiable, by Schwarz's lemma  we will always have
$$
v_{i_1,\dots,i_k}=v_{\pi(i_1),\dots,\pi(i_k)}
$$
for any permutation $\pi$ of the subscripts.
\\ \indent
For a $d$-dimensional vector $\ix$ and a (vector or scalar)-valued function $v$, the notation $v(\tilde{\ix}_i)$ means that we fix all coordinates except one, i.e. we consider
$$
v(x_1,\dots, x_{i-1}, \cdot, x_{i+1}, \dots, x_d).
$$
Accordingly, the notation $\int {\cdot\ d\tilde{\ix}_i}$ means that we integrate over the variables
$$
x_1,\dots, x_{i-1}, x_{i+1}, \dots, x_d,
$$
for a fixed value of $x_i$. For shortness, a function $v(t,\cdot)$ is denoted by $v(t)$. The norm of an $N$-dimensional vector $\ve$ is defined by
$$
|\ve|=\sqrt{v_1^2+\dots+v_{N}^2}.
$$
It should not be confused with the $L_2$ norm of a \textit{function}, which will be introduced in the next subsection and is also denoted by $| \cdot |$: the meaning of the notation will always be clear from the context. We use the notation $g^{-}=\max(-g,0)$ and $g^{+}=\max(g,0)$.

\subsection{Sobolev spaces} \label{sob}

For $N,d' \geq 1$, consider an integrable $\R^N$-valued function $v$ on $\T^{d'}$.We only study \textit{spatial} Sobolev norms for functions considered \textit{at a fixed moment of time}.
We do not always assume that $d'=d$: for instance, we will study functions of the type $v(\tilde{\ix}_i)$ which are defined on $\T^1$. The dimensions $N,d'$ are always clear from the context, and thus are not specified in the notation for Sobolev norms.
\\ \indent 
For $p \in [1,\infty]$, we denote the Lebesgue $L_p$ norm of a scalar-valued function $v$ by $\left|v\right|_p$. For a vector-valued function $\ve$, we define this norm as the norm in $L_p$ of the function $|\ve|$, and denote it by $\left|\ve\right|_p$ . We denote the $L_2$ norm by $|\cdot|$, and the corresponding scalar product by $\left\langle \cdot,\cdot\right\rangle$. From now on $L_p,\ p \in [1,\infty]$ denotes the space of functions in $L_p(\T^{d'})$. Similarly, $C^{\infty}$ is the space of $C^{\infty}$-smooth functions on $\T^{d'}$.
\\ \indent
Except in Appendix 1, we only study Sobolev norms for zero mean functions. Thus, in the following, we always assume that $\int_{\T^{d'}}{v}=0$. In particular, we never study directly the Sobolev norms of the potential $\psi$: either we consider the mean value function $\psi-\int{\psi}$ or the partial derivatives of $\psi$.
\\ \indent
For a nonnegative integer $m$ and $p \in [1,\infty]$, $W^{m,p}$ stands for the Sobolev space of zero mean functions $v$ on $\T^{d'}$ with finite homogeneous norm
\begin{equation} \label{Wmp}
\left|v\right|_{m,p}=\sum_{|\boldsymbol\alpha|=m}{\frac{m!}{\alpha_1! \dots \alpha_k!} \left|\Big( \frac{d^{m} v_1}{d \ix^{\boldsymbol\alpha}}, \dots, \frac{d^{m} v_N}{d \ix^{\boldsymbol\alpha}} \Big) \right|_p}.
\end{equation}
Here and from now on, $|\boldsymbol\alpha|$ denotes the norm of the multi-index
$$
\boldsymbol\alpha=(\alpha_1,\dots,\alpha_{d'}),
$$
defined by
$$
\alpha_1+\dots+\alpha_{d'}.
$$
In particular, $W^{0,p}=L_p$ for $p \in [1,\infty]$. For $p=2$, we denote $W^{m,2}$ by $H^m$, and abbreviate the corresponding norm as $\left\|v\right\|_m$. 
\\ \indent
We recall a version of the classical Gagliardo--Nirenberg inequality (see \cite[Appendix]{DG95}). We will refer to this inequality as (GN).
\begin{lemm} \label{GN}
For a smooth zero mean function $v$ on $\T^{d'}$, we have 
$$
\left|v\right|_{\beta,r} \leq C \left|v\right|^{\theta}_{m,p} \left|v\right|^{1-\theta}_{q},
$$
where $m>\beta\geq 0$, and $r$ is defined by
$$
\frac{d'}{r}=\beta-\theta \Big( m-\frac{d'}{p} \Big)+(1-\theta)\frac{d'}{q},
$$
under the assumption $\theta=\beta/m$ if $m-|\beta|-d'/p$ is a nonnegative integer, and $\beta/m \leq \theta < 1$ otherwise. The constant $C$ depends on $m,p,q,\beta,\theta,d'$.
\end{lemm}
Let us stress that we only use this inequality in cases when it gives the same value of $\theta$ as in 1d. Actually, the only place where we use it in a multi-d setting is when we mention that the proof of Lemma~\ref{lmubuinfty} is word-to-word the same as in 1d.
\\ \indent
We will use a norm denoted by $|\cdot|_{m,p}^{\sim}$, which is defined for $m \geq 0,\ p \in [1,+\infty)$ and is equivalent to the norm $|\cdot|_{m,p}$ defined above. For its definition, see Corollary~\ref{linalgcor}. By analogy with the notation $\Vert \cdot \Vert_m$, we will abbreviate as $\Vert \cdot \Vert^{\sim}_m$ the norm $|\cdot|_{m,2}^{\sim}$.
\\ \indent
For any $s \geq 0$, we define $H^{s}$ as the Sobolev space of zero mean functions $v$ on $\T^{d'}$ with finite norm
\begin{equation} \label{Sobolevspectr}
\left\| v \right\|'_{s}= \Big( \langle v,\ (-\Delta)^s v \rangle \Big)^{1/2}=(2 \pi)^{s} \Big( \sum_{\en \in \Z^{d'}}{|\en|^{2s} |\hat{v}(\en)|^2} \Big)^{1/2},
\end{equation}
where $\hat{v}(\en)$ are the complex Fourier coefficients of $v(x)$. For integer values of $s=m$, this norm is equivalent to the previously defined $H^m$ norm $\left\| \cdot \right\|$. For $s \in (0,1)$, $\left\|v\right\|'_{s}$ is equivalent to the norm
\begin{equation} \label{Sobolevfrac}
\left\|v\right\|^{''}_{s}=\Bigg( \int_{\ix \in \T^{d'},\ |\er| \leq 1}{\frac{|v(\ix+\er)-v(\ix)|^2}{|\er|^{2s+d'}}}\ d \ix\ d \er \Bigg)^{1/2}.
\end{equation}
Moreover, for all integers $m \geq 0$ we have the embedding
\begin{equation} \label{Sobinj}
|v|_{m,\infty} \leq C(s) \left\|v\right\|^{'}_{m+s},\ s>d/2
\end{equation}
(see \cite{Ada75,Tay96}).
\\ \indent
Finally, it should be noted that the integer $s_0(d)$, defined by:
\begin{equation} \label{s0}
s_0=(d+1)/2\ if\ d\ even;\ d/2+1\ if\ d\ odd
\end{equation}
plays a crucial role in the study of the well-posedness for (\ref{HJint}) (see Section~\ref{prel} and Appendix 1).

\subsection{Random setting} \label{rand}

We provide each space
$$
W^{m,p}(\T^d),\ m \geq 0,\ p \in [1,\infty]
$$
of scalar-valued functions with the Borel $\sigma$-algebra. Then we consider a random process $w(t)=w^{\omega}(t),\ \omega \in \Omega,\ t \geq 0$, valued on the space of zero mean value functions in $L_2(\T^d)$ and defined on a complete probability space $(\Omega, \F, \Pe)$. We suppose that $w(t)$ defines a smooth in space Wiener process with respect to a filtration $\F_t,\ t \geq 0$, in each space $W^{m,p}(\T^d),\ m \geq 0,\ p \in [1,\infty]$. Moreover, we assume that the process $w(t)$ is \textit{diagonal} in the sense that its projections on the Fourier modes are independent weighted Wiener processes. In other words, we assume the following:
\medskip
\\ \indent
\textbf{i)}\ The process $w(t)$ can be written as
\begin{equation} \label{Itodiag}
w(t,\ix)=\sum_{\en \in \widetilde{\Z^d}}({a_{\en} w_{\en}(t) \cos(2 \pi \en \cdot \ix) + b_{\en} \tilde{w}_{\en}(t) \sin(2 \pi \en \cdot \ix))},
\end{equation}
where 
\begin{align} \nonumber
\widetilde{\Z^d}=&\lbrace \en \in \Z^d\ |\ n_1>0 \rbrace \cup \lbrace \en \in \Z^d\ |\ n_1=0, n_2>0 \rbrace
\dots 
\\ \nonumber
&\cup \lbrace \en \in \Z^d\ |\ n_1=0, \dots, n_{d-1}=0,n_d>0 \rbrace,
\end{align}
$w_{\en}$, $\tilde{w}_{\en}$ are independent Wiener processes and for any $k>0$ we have $a_{\en},b_{\en}=o(|\en|^{-k})$. Without loss of generality, we can assume that for all $\en$, we have $a_{\en},b_{\en} \geq 0$.
\\ \indent
\textbf{ii)}\ The process $w(t)$ is non-trivial: in other words, at least one of the coefficients $a_{\en},\ b_{\en}$ is not equal to $0$.
\medskip
\\ \indent
Thus, for $\zeta,\chi \in L_2,$
$$
\E(\left\langle w(s),\zeta\right\rangle \left\langle w(t),\chi \right\rangle)=\frac{1}{2} \min(s,t) \left\langle Q\zeta,\chi\right\rangle,
$$
where $Q$ is the correlation operator defined by
$$
Q(\cos(2 \pi \en \cdot \ix))=a_n^2 \cos(2 \pi \en \cdot \ix);\ Q(\sin(2 \pi \en \cdot \ix))=b_n^2 \sin(2 \pi \en \cdot \ix),
$$
which defines a continuous mapping from $L_2(\T^d)$ into $H^m(\T^d)$ for each $m \geq 0$.
\\ \indent
Note that since we have $w(t) \in C^{\infty}$ for every $t$, a.s., we can redefine the Wiener process $w$ so that this property holds for all $\omega \in \Omega$.
We will denote $w(t)(x)$ by $w(t,x)$. For more details on the construction of infinite-dimensional Wiener processes, see \cite[Chapter 4]{DZ92}.
\\ \indent
For $m \geq 0$, we denote by $I_m$ the quantity
$$
I_m=Tr_{H^m}(Q)=\E \left\|w(1)\right\|_m^2.
$$
From now on, the term $dw(s)$ denotes the stochastic differential corresponding to the Wiener process $w(s)$ in the space $L_2$.
\\ \indent
Now fix $m \geq 0$. By Fernique's Theorem \cite[Theorem 3.3.1]{Kuo75}, there exist $\iota_{m},C_m>0$ such that for $T \geq 0$,
\begin{equation} \label{Fern}
\E \exp \Big(\iota_m \left\|w(T)\right\|_m^2/T \Big) \leq C_m.
\end{equation}
Therefore by Doob's maximal inequality for infinite-dimensional submartingales \cite[Theorem 3.8. (ii)]{DZ92} we have the following inequality, which holds uniformly in $\tau \geq 0$:
\begin{align} \label{moments}
\E \sup_{t \in [\tau,T+\tau]} {\left\|w(t)-w(\tau)\right\|^k_m} &\leq \Big( \frac{k}{k-1} \Big)^k \E \left\|w(T+\tau)-w(\tau)\right\|_m^k
\\ \nonumber
&=C(m,k) T^{k/2} < +\infty,
\end{align}
for any $T>0$ and $1<k<\infty$.
\\ \indent
Note that the estimates in this subsection still hold for the successive spatial derivatives of $w$, which are also smooth in space infinite-dimensional Wiener processes.

\subsection{Preliminaries} \label{prel}

\indent
We begin by considering the viscous Hamilton-Jacobi equation (\ref{HJbegin}). Here, $t \geq 0$, $\ix \in \T^d=(\R/\Z)^d$ and the viscosity coefficient satisfies $\nu \in (0,1]$. The function $f$ is strongly convex, i.e. it satisfies (\ref{strconvex}), and $C^{\infty}$-smooth. We also assume that for any $m \geq 0$ the $m$-th partial derivatives of $f$ satisfy
\begin{equation} \label{poly}
\exists h \geq 0,\ C_m>0:\ \Big|\sum_{|\boldsymbol\alpha|=m}{\frac{\partial^m f(\ix)}{\partial \ix^{\boldsymbol\alpha}}} \Big| \leq C_m (1+|\ix|)^h,\quad \ix \in \R^d,
\end{equation}
where $h=h(m)$ is a function such that $1 \leq h(1) < 2$ (the lower bound on $h(1)$ follows from (\ref{strconvex})). The usual Burgers equation corresponds to $f(\ix)=|\ix|^2/2$.
\\ \indent
The white-forced generalised Hamilton-Jacobi equation is (\ref{HJbegin}) with the random forcing term
$$
\eta^{\omega}=\partial w^{\omega}/\partial t,
$$
added on the right-hand side. Here, $w^{\omega}(t),\ t \geq 0$ is the Wiener process with respect to $\F_t$ defined above.

\begin{defi} \label{weak}
We say that an $H^{s_0(d)}$-valued process $\ue(t,x)=\ue^{\omega}(t,x)$ (for the definition of $s_0$ see (\ref{s0})) is a solution of the equation
\begin{equation} \label{whiteBurgers}
\begin{cases}
\frac{\partial \ue}{\partial t} +  (\nabla f(\ue) \cdot \nabla) \ue  - \nu \Delta \ue = \nabla \eta^{\omega}
\\
\ue=\nabla \psi
\end{cases}
\end{equation}
if for every $t \geq 0$ and for every $\omega \in \Omega$, $\ue=\nabla \psi$, where $\psi$ satisfies the following properties: 
\\ \indent
i) For $t \geq 0$, $\omega \mapsto \psi^{\omega}(t)$ is $\F_t$-measurable.
\\ \indent
ii) The function $t \mapsto \psi^{\omega}(t)$ is continuous in $H^{s_0+1}$ (its gradient $t \mapsto \ue^{\omega}(t)$ is therefore continuous in $H^{s_0}$) and $\psi^{\omega}$ satisfies
\begin{align} \label{HJint}
\psi^{\omega}(t)=\psi^{\omega}(0)-\int_{0}^{t}{\Big( \nu L \psi^{\omega}(s)+f(\nabla \psi^{\omega})(s) \Big) ds}+w^{\omega}(t),
\end{align}
where $L=-\Delta$.
\end{defi}
\indent
As a corollary of this definition, we obtain that $\ue$ satisfies
\begin{align} \label{Burgersint}
\ue^{\omega}(t)=\ue^{\omega}(0)-\int_{0}^{t}{\Big( \nu L \ue^{\omega}(s)+\frac{1}{2} B(\ue^{\omega})(s) \Big) ds}+\nabla w^{\omega}(t),
\end{align}
where $B(\ue)=2 (\nabla f(\ue) \cdot \nabla) \ue$. When studying solutions of (\ref{whiteBurgers}), we always assume that the initial potential $\psi^0=\psi(0,\cdot)$ is $C^{\infty}$-smooth.
\\ \indent
For a given initial condition, (\ref{HJint}), and therefore (\ref{whiteBurgers}), has a unique solution, i.e. any two solutions coincide for all $\omega \in \Omega$. For shortness, this solution (resp., the corresponding potential) will be denoted by $\psi$ (resp., $\ue$). To prove this, we can use the same arguments as in 1d (cf. \cite{BorPhD}). Namely, to prove local well-posedness we use the \enquote{mild solution} technique (cf. \cite[Chapter 14]{DZ96}) and a bootstrap argument. Finally, global well-posedness follows from uniform bounds of the same type as in Section~\ref{upper}. For more details, see Appendix 1.
\\ \indent
Once $\ue^0=\nabla \psi^0$ is fixed, $\psi^0$ is fixed up to an additive constant. Moreover, if we consider two different initial conditions $\psi^0$ and $\psi^0+C$, then the difference between the corresponding solutions to (\ref{HJint}) will always be equal to $C$. In other words, fixing $\ue^0$ is equivalent to fixing an equivalence class of initial conditions $\psi^0+C,\ C \in \R$.
\\ \indent
Since the forcing and the initial condition are smooth in space, we can also show that $t \mapsto \ue(t)$ is time-continuous in $H^m$ for every $m \geq s_0$ and the spatial derivatives of $t \mapsto \psi(t)- w(t)$ are in $C^{\infty}$ for all $t$. Consequently, $\psi$ is also a strong solution of the equation
\begin{equation} \label{HJ}
\frac{\partial (\psi-w)}{\partial t} + f(\nabla \psi)  - \nu \Delta \psi =0,
\end{equation}
and $\ue^0$ is a strong solution of the equation
\begin{equation} \label{Bur}
\frac{\partial (\ue-\nabla w)}{\partial t}+(\nabla f(\ue) \cdot \nabla) \ue -\nu \Delta \ue= 0.
\end{equation}
Solutions of (\ref{whiteBurgers}) make a time-continuous Markov process in $H^{s_0}$. For details, we refer to \cite{KuSh12}, where the stochastic 2D Navier-Stokes equations are studied in a similar setting.
\\ \indent
Now consider, for a solution $\ue(t,x)$ of (\ref{whiteBurgers}), the functional
$$
G_m(\ue(t))=\left\|\ue(t)\right\|_m^2
$$
and apply It{\^o}'s formula \cite[Theorem 4.17]{DZ92} to (\ref{Burgersint}). We get
\begin{align} \nonumber
\left\|\ue(t)\right\|_m^2=& \left\|\ue^0\right\|_m^2 - \int_{0}^{t}{\left( 2\nu \left\|\ue(s)\right\|_{m+1}^2 +\langle L^m \ue(s),\ B(\ue)(s) \rangle\right) ds}
\\ \label{Itoexpint}
 &+ 2 \int_{0}^{t}{\langle L^m \ue(s),\ d \nabla  w(s)\rangle}+ I_{m+1} t.
\end{align}
We recall that for $m \geq 0$, $I_m=Tr_{H^{m}}(Q)$. Consequently,
\begin{align} \label{Itoexpdiff}
&\frac{d}{dt} \E \left\|\ue(t)\right\|_m^2 = -2 \nu \E \left\|\ue(t)\right\|_{m+1}^2 - \E\ \langle L^m \ue(t),\ B(\ue)(t) \rangle+I_{m+1}.
\end{align}

\subsection{Agreements} \label{agree}

From now on, all constants denoted by $C$ with an eventual subscript are positive and nonrandom. Unless otherwise stated, they depend only on $f$ and on the distribution of the Wiener process $w$. By $C(a_1,\dots,a_k)$ we denote constants which also depend on the parameters $a_1,\dots,a_k$. By $X \overset{a_1,\dots,a_k}{\lesssim} Y$ we mean that $X \leq C(a_1,\dots,a_k) Y$. The notation $X \overset{a_1,\dots,a_k}{\sim} Y$ stands for
$$
Y \overset{a_1,\dots,a_k}{\lesssim} X \overset{a_1,\dots,a_k}{\lesssim} Y.
$$
In particular, $X \lesssim Y$ and $X \sim Y$ mean that $X \leq C Y$ and $C^{-1} Y \leq X \leq C Y$, respectively. All constants are independent of the viscosity $\nu$ and of the initial value $\ue^0$.
\\ \indent
We denote by $\ue=\ue(t,x)$ the solution to (\ref{whiteBurgers}) with an initial condition $\ue^0=\nabla \psi^0$ and by $\psi$ the corresponding solution to (\ref{HJint}), which is, for a given value of $\ue^0$, uniquely defined up to an additive constant.
\\ \indent
For simplicity, in Sections~\ref{upper}-\ref{turb}, we assume that $\ue^0$ is deterministic. However, we can easily generalise all results in these sections to the case of a random initial condition $\ue^0(\omega)$  independent of $w(t),\ t \geq 0$. Indeed, in this case for any measurable functional $\Phi(\ue(\cdot))$ we have
$$
\E \Phi(\ue(\cdot))=\int{\E \Big(\Phi(\ue(\cdot))\ |\ \ue^0=\ue^0(\omega) \Big) d \mu (\ue^0(\omega))},
$$
where $\mu(\ue^0(\omega))$ is the law of the r.v. $\ue^0(\omega)$.
\\ \indent
For $\tau \geq 0$ and $\ue(\tau)$ independent of $w(t)-w(\tau),\ t \geq \tau$, the Markov property yields that
$$
\E \Phi(\ue(\tau+\cdot))=\int{\E \Big(\Phi(\ue(\cdot))\ |\ue^0=\ue^{\omega}(\tau) \Big) d \mu (\ue^{\omega}(\tau))}.
$$
Consequently, all $\ue^0$-independent estimates which hold for time $t$ or a time interval $[t,t+T]$ actually hold for time $t+\tau$ or a time interval $[t+\tau,t+\tau+T]$, uniformly in $\tau \geq 0$. Thus, for $T \geq 0$, to prove a $\ue^0$-independent  estimate which holds uniformly for $t \geq T$, it suffices to consider the case $t=T$.
%
%

\subsection{Setting and notation in Section~\ref{turb}} \label{agreeturb}

Consider an observable $A$, i.e. a real-valued functional on a Sobolev space $H^m$, which we evaluate on the solutions $u^{\omega}(s)$. We denote by $\lbrace A \rbrace$ the average of $A(u^{\omega}(s))$ in ensemble and in time over $[t,t+T_0]$:
$$
\lbrace A \rbrace=\frac{1}{T_0}\ \int_{t}^{t+T_0}{\E\ A(u^{\omega}(s))\ ds},\ t \geq T_1=T_0+2.
$$
The constant $T_0$ is the same as in Theorem~\ref{avoir}. In this section, we assume that $\nu \leq \nu_0$, where $\nu_0$ is a positive constant.  Next, we define the intervals
\begin{equation} \label{ranges}
J_1=(0,\ C_1 \nu];\ J_2=(C_1 \nu,\ C_2];\ J_3=(C_2,\ 1].
\end{equation}
For the value of $\nu_0$, $C_1$ and $C_2$, see (\ref{nu0eq}). In other words, $J_1 = \lbrace \ell:\ 0 < \ell \lesssim \nu \rbrace$, $J_2 = \lbrace \ell:\ \nu \lesssim \ell \lesssim 1 \rbrace$, $J_3=\lbrace \ell:\ \ell \sim 1\rbrace$. 
\\ \indent 
In terms of the Kolmogorov 1941 theory \cite{Fri95}, the interval $J_1$ corresponds to the \textit{dissipation range}. In other words, for the Fourier modes $\en$ such that $|\en|^{-1} \preceq C_1 \nu$, the expected values of the Fourier coefficients $\lbrace|\hat{\ue}(\en)|^2\rbrace$ decrease super-algebraically in $|\en|$. The interval $J_2$ corresponds to the \textit{inertial range}, where quantities such as the (layer-averaged) \textit{energy spectrum} $E(k)$ defined by
\begin{equation} \label{spectrum}
E(k)= k^{-1} \sum_{|\en| \in [M^{-1}k,Mk]}{ \{ |\hat{\ue}(\en)|^2 \} }
\end{equation}
behave as a negative degree of $k$. Here $M \geq 2$ is a large enough constant (cf. the proof of Theorem~\ref{spectrinert}). The boundary $C_1 \nu$ between these two ranges is the \textit{dissipation length scale}. Finally, the interval $J_3$ corresponds to the \textit{energy range}, i.e. the sum $\sum \lbrace| \hat{\ue}(\en)|^2\rbrace$ is mostly supported by the Fourier modes corresponding to $|\en|^{-1} \in J_3$.
\\ \indent
The positive constants $C_1$ and $C_2$ can take any value provided that
\begin{equation} \label{C1C2}
C_1 \leq \frac{1}{4} K^{-2};\quad 5 K^2 \leq \frac{C_1}{C_2}<\frac{1}{\nu_0}.
\end{equation}
Here, $K$ is a positive constant (see (\ref{K})). Note that the intervals defined by (\ref{ranges}) are non-empty and do not intersect each other for all values of $\nu \in (0,\nu_0]$, under the assumption (\ref{C1C2}). 
\\ \indent
The constants $C_1$ and $C_2$ can be made as small as desired. On the other hand, by (\ref{spectrinertupper2}) the ratio
\begin{equation} \nonumber
\frac{\sum_{|\en|^{-1} \in J_3} |\hat{\ue}(\en)|^2}{\sum_{\en \in \Z^d} |\hat{\ue}(\en)|^2}
\end{equation}
tends to $1$ as $C_2$ tends to $0$, uniformly in $\nu$. This allows us to choose $C_2$ so that
$$
\sum_{|\en| < C_2^{-1}}{\lbrace|\hat{\ue}(\en)|^2\rbrace} \geq \frac{99}{100} \sum_{\en \in \Z^d}{\lbrace|\hat{\ue}(\en)|^2\rbrace}.
$$
Now we suppose that $\er \in \R^d$, $p,\alpha \geq 0$. We consider the averaged \textit{directional} increments
\begin{align} \nonumber
S_{p,\alpha,i}(\er)=& \Big\{ \Big( \int_{\ix \in \T^d }{|\psi_i(\ix+\er)-\psi_i(\ix)|^p d \ix} \Big)^{\alpha} \Big\},\ 1 \leq i \leq d,
\end{align}
the averaged \textit{longitudinal} increments
\begin{align} \nonumber
S^{\|}_{p,\alpha}(\er)=& \Bigg\{ \Bigg( \int_{\ix \in \T^d }{\Bigg| \frac{(\ue(\ix+\er)-\ue(\ix)) \cdot \er}{|\er|} \Bigg|^p d \ix} \Bigg)^{\alpha} \Bigg\}
\end{align}
and the averaged increments
\begin{align} \nonumber
S_{p,\alpha}(\er)=& \Big\{ \Big( \int_{\ix \in \T^d }{|\ue(\ix+\er)-\ue(\ix)|^p d \ix} \Big)^{\alpha} \Big\}.
\end{align}
\indent
Now, for $0< \ell \leq 1$, we define the averaged moments of the space increments on the scale $\ell$ for the flow $\ue(t,x)$:
\begin{align} \nonumber
&S_{p,\alpha}(\ell)= c_d^{-1} \ell^{-(d-1)} \int_{\er \in \ell \mathcal{S}^{d-1} }{S_{p,\alpha}(\er) d \sigma(\er)},
\end{align}
where $d\sigma$ stands for the surface measure on $\ell \mathcal{S}^{d-1}$ and $c_d$ is the surface of $\mathcal{S}^{d-1}$. The quantity $S_{p,1}(\ell)$ is denoted by $S_p(\ell)$; it corresponds to the \textit{structure function} of order $p$, while the flatness $F(\ell)$, given by
\begin{equation} \label{flatness}
F(\ell)=S_4(\ell)/S_2^2(\ell),
\end{equation}
measures spatial intermittency at the scale $\ell$ \cite{Fri95}.

\section{Main results} \label{results}

In Sections~\ref{upper}-\ref{main}, we prove sharp upper and lower estimates for a large class of Sobolev norms of $\ue$. A key result is proved in Theorem~\ref{uxpos}. Namely, there we obtain that for $k \geq 1$,
\begin{equation} \label{uxposresults}
\E\ \big( \max_{s \in [t,t+1],\ \ix \in \T^d} \max_{1 \leq i \leq d} \psi_{ii}(s,\ix) \big)^{k}\overset{k}{\lesssim} 1, \quad t \geq 1.
\end{equation}
This result is generalised to all second derivatives  $(\ka \cdot \nabla)^2 \psi$, where $\ka$ is a vector with integer coefficients, in Lemma~\ref{uxposbis}.
\\ \indent
The main estimates for Sobolev norms are contained in the first part of Theorem~\ref{avoir}, where we prove that for $m=0$ and $p \in [1,\infty]$, $m=1$ and $p \in [1,\infty)$, or $m \geq 2$ and $p \in (1,\infty)$, we have
\begin{equation} \label{avoirresults}
\Big\{ \left|\ue(s)\right|_{m,p}^{\alpha} \Big\}^{1/\alpha} \overset{m,p,\alpha}{\sim} \nu^{-\gamma},\quad \alpha>0,\ t,T \geq T_0,
\end{equation}
where $\gamma=\max (0,m-1/p)$ and $T_0$ is a constant. In particular, this result implies that
$$
\frac{\{ \left\|\ue(s)\right\|_{m} \}}{\{  \left\|\ue(s)\right\|_{m+1}  \}} \sim \nu,\ m \geq 1.
$$
In the language of the turbulence theory, the characteristic dissipaton scale of the flow is therefore of the order $\nu$.
\\ \indent
For $m \geq 1$ and $p=\infty$, we do not have a similar bound, as opposed to the 1d case. However, there exists $1 \leq i \leq d$ such that we have the following result, which will play a key role when we estimate the small-scale quantities:
\begin{equation} \label{avoirlinftyresults}
\Big\{ \int_{\tilde{\ix}_i}  | \psi_i(s,\tilde{\ix}_i)|_{m,\infty}^{\alpha} \Big\}^{1/\alpha} \overset{m,\alpha}{\sim} \nu^{-m},\ \alpha>0,\ t,T \geq T_0.
\end{equation}
This result tells us that in average, the 1d restrictions of $\psi_i$ (which is, as we recall, the $i$-th coordinate of $\ue$) for fixed values of $\tilde{\ix}_i$ have the same behaviour as the 1d solutions (see \cite{BorW}).
\\ \indent
In Section~\ref{turb} we obtain sharp estimates for analogues of the quantities characterising the hydrodynamical turbulence. Although we only prove results for quantities averaged over a time period of length $T_0$, those results can be immediately generalised to quantities averaged over time periods of length $T \geq T_0$. 
\\ \indent
We assume that $\nu \in (0,\nu_0]$, where $\nu_0 \in (0,1]$ is a constant. As the first application of the estimates (\ref{uxposresults}-\ref{avoirlinftyresults}), in Section~\ref{turb} we obtain sharp estimates for the quantities $S_{p,\alpha},\ \alpha \geq 0$. Namely, by Theorem~\ref{avoir2}, for $\ell \in J_1$:
$$
\quad \ \ \ S_{p,\alpha}(\ell) \overset{p,\alpha}{\sim} \left\lbrace \begin{aligned} & \ell^{\alpha p},\ 0 \leq p \leq 1. \\ & \ell^{\alpha p} \nu^{-\alpha(p-1)},\ p \geq 1, \end{aligned} \right.$$
and for $\ell \in J_2$:
$$
S_{p,\alpha}(\ell) \overset{p,\alpha}{\sim} \left\lbrace \begin{aligned} & \ell^{\alpha p},\ 0 \leq p \leq 1. \\ & \ell^{\alpha},\ p \geq 1. \end{aligned} \right.
$$
Consequently, for $\ell \in J_2$ the flatness function $F(\ell)=S_4(\ell)/S_2^2(\ell)$ satisfies $F(\ell) \sim \ell^{-1}.$ Thus, solutions $\ue$ are highly intermittent in the inertial range.
\\ \indent
On the other hand, we obtain estimates for the spectral asymptotics of Burgulence. Namely, for all $m \geq 1$ and $\en \in \Z^d-\lbrace \zero \rbrace$, (\ref{superalg}) tells us that
$$
\lbrace |\hat{\ue}(\en)|^2 \rbrace \overset{m}{\lesssim} (\nu |\en|)^{-2m} \nu,
$$
and for $k$ such that $k^{-1} \in J_2$, by Remark~\ref{spectrinertrmq} we get
$$
\Bigg\{ \Bigg( k^{-1} \sum_{| \en | \in [M^{-1}k,Mk]}{|\hat{\ue}(\en)|^2} \Bigg)^{\alpha} \Bigg\} \overset{\alpha}{\sim} k^{-2\alpha},\quad \alpha>0.
$$
In particular, in the inertial range the energy spectrum satisfies $E(k) \sim k^{-2}$.
\\ \indent
Finally, in Section~\ref{stat} we prove that (\ref{whiteBurgers}) has a unique stationary measure $\mu$, and we give an estimate of the speed of convergence to it. Then we deduce that all estimates listed above still hold if we replace the brackets $\lbrace \cdot \rbrace$ with averaging with respect to $\mu$.

\section{Upper estimates for Sobolev norms} \label{upper}


\begin{rmq} \label{upperrmq}
In all results in Sections~\ref{upper}-\ref{turb}, quantities estimated for fixed $\omega$, such as maxima in time of Sobolev norms or
$$
\max_{s \in [t,t+1],\ \ix \in \T^d}\ \max_{1 \leq i \leq d} {\psi_{ii}(s,\ix)}
$$
can be replaced by their suprema over all smooth initial conditions. For instance, the quantity
$$
\max_{s \in [t,t+1]} |\ue(s)|_{m,p}
$$
can be replaced by
$$
\sup_{\ue^0 \in C^{\infty}} \max_{s \in [t,t+1]} |\ue(s)|_{m,p}.
$$
For the lower estimates, this fact is obvious. For the upper ones, the reason is that these quantities admit upper bounds of the form
$$
(1+\max_{s \in [t-\tau,t+\tau]}{\Vert w(s) \Vert_m})^{\alpha(m)} \nu^{-\beta(m)}.
$$
For more details, see \cite[Section 3.3]{BorW}, where this fact is proved in the 1d setting: the proof of the general case is word-to-word the same.
\end{rmq}

The following theorem is proved using a stochastic version of the Kruzhkov maximum principle \cite{Kru64}. The main idea is that if there was no forcing, then the partial derivatives $\psi_{ii}$ would be bounded from above by $C/t$. Since the Wiener process $w$ is smooth in space, we have good (i.e. uniform with respect to the initial condition) upper bounds for those derivatives. In this theorem and in the following ones, the potential assumption is crucial. Without it, the maximum principle still holds for the components $u_i$. However, we do not have any estimates which are uniform in time in the forced nonpotential case.

\begin{theo} \label{uxpos}
Denote by $X_{i,\tau}$ the r.v.'s
$$
X_{i,\tau}=\max_{s \in [\tau,\tau+1],\ \ix \in \T^d} \psi_{ii}(s,\ix),\ 1 \leq i \leq d.
$$
For $k \geq 1$, we have
$$
\E \ (\max_{1 \leq i \leq d} X_{i,\tau})^{k} \overset{k}{\lesssim} 1,\quad \tau \geq 1.
$$
\end{theo}

\begin{proof}
It suffices to consider the case $\tau=1$. For simplicity, we will denote $X_{i,\tau}$ by $X_i$.
\\ \indent
By (\ref{moments}), it suffices to prove the lemma's statement with $\psi_{ii}$ replaced by $\psi_{ii}-w_{ii}$. Consequently, it suffices to prove the lemma's statement with $[1,2]$ replaced by $[0,2]$ and $X_i$ replaced by $Y_i$, where $Y_i$ denotes
$$
Y_{i}=\max_{s \in [0,2],\ \ix \in \T^d}{s(\psi_{ii}(s,\ix)-w_{ii}(s,\ix))},\ 1 \leq i \leq d.
$$
Without loss of generality, we can assume that the maximum of $\psi_i$ is reached for $i=1$. We will denote $Y_1$ by $N$.
\\ \indent
Consider the equation (\ref{HJ}). Differentiating twice in $x_1$, we obtain that
\begin{align} \nonumber
&(\psi_{11}-w_{11})_t+\sum_{1 \leq i,j \leq d}{\psi_{1i} \psi_{1j} f_{ij}(\nabla \psi)}+\sum_{1 \leq i \leq d}{(\psi_{11})_i f_i(\nabla \psi)} =\nu \Delta \psi_{11}.
\end{align}
Putting $v=\psi-w$ and using (\ref{strconvex}), we get
\begin{align} \nonumber
&(v_{11})_t + \sigma |\nabla \psi_1|^2 + \sum_{1 \leq i \leq d}{(v_{11})_i f_i(\nabla \psi)} + \sum_{1 \leq i \leq d}{w_{11i} f_i(\nabla \psi)}
\\ \label{v11}
&\leq \nu \Delta v_{11}+\nu \Delta w_{11}.
\end{align}
Consider $\tilde{v}(t,x)=tv_{11}(t,x)$ and multiply (\ref{v11}) by $t^2$. For $t>0$, $\tilde{v}$ satisfies
\begin{align} \nonumber
&t\tilde{v}_t -\tilde{v} +\sigma (\tilde{v}+tw_{11})^2 + t \sum_{1 \leq i \leq d}{\tilde{v}_i f_i(\nabla \psi)}  + t^2 \sum_{1 \leq i \leq d}{w_{11i} f_i(\nabla \psi)}
\\ \label{tildev}
& \leq \nu t \Delta \tilde{v}+\nu t^2 \Delta w_{11}.
\end{align}
Now observe that if the zero mean function $\tilde{v}$ does not vanish identically on the domain $S=\left[0,2\right] \times \T^d$, then it attains its positive maximum $N$ on $S$ at a point $(t^0,\ix^0)$ such that $t^0>0$. At such a point, we have $\tilde{v}_t \geq 0$, $\tilde{v}_i=0$ for all $i$ and $\Delta \tilde{v} \leq 0$. By (\ref{tildev}), at $(t^0,\ix^0)$ we have the inequality
\begin{equation} \label{maxpoint}
\sigma (\tilde{v}+tw_{11})^2 \leq \nu t^2 \Delta w_{11}+\tilde{v} - t^2 \sum_{1 \leq i \leq d}{w_{11i} f_i(\nabla \psi)}.
\end{equation}
Now denote by $K$ the r.v.
$$
K=\max_{t \in [0,2]} |w(t)|_{4,\infty}
$$
and by $\delta$ the quantity
$$
\delta=2-h(1).
$$
(see (\ref{poly})). Since $\delta>0$, we get
\begin{align} \nonumber 
\Big| t^2 \sum_{1 \leq i \leq d}{w_{11i} f_i(\nabla \psi)} \Big| & \leq C K t^{\delta} (t+t|\nabla \psi|)^{2-\delta}
 \\ \nonumber
& \leq C K t^{\delta} (t+t |\nabla (\psi-w)|+t |\nabla w|)^{2-\delta}.
\end{align}
Since $N=\min_{1 \leq i \leq d} Y_i$ is reached for $i=1$ and for every $i$ and $\tilde{\ix}_i$, $t (\psi_i-w_i)(\tilde{\ix}_i)$ is the zero mean primitive of $t (\psi_{ii}-w_{ii})(\tilde{\ix}_i)$, at $(t^0,\ix^0)$ we have the inequality
\begin{align} \nonumber 
& \Big| t^2 \sum_{1 \leq i \leq d}{w_{11i} f_i(\nabla \psi)} \Big| \leq C K (1+N+K)^{2-\delta}.
\end{align}
From now on, we assume that $N \geq 2K$. Since $\nu \in (0,1]$, the relation (\ref{maxpoint}) yields
\begin{equation} \label{KN}
\sigma (N-2K)^2 \leq 4K+N+C K (1+N+K)^{2-\delta}.
\end{equation}
Consequently, if $N \geq 2K$, then $N \leq C(K+1)^{1/\delta}$. Since by (\ref{moments}) all moments of $K$ are finite, all moments of $N$ are also finite. This proves the lemma's assertion.
\end{proof}


\begin{cor} \label{Lpupper}
For $k \geq 1$,
$$
\E \max_{s \in [t,t+1]} \left|\ue(s)\right|^k_{p} \overset{k}{\lesssim} 1,\quad p \in [1,\infty],\ t \geq 1.
$$
\end{cor}

\begin{proof}
It suffices to consider the case $p=\infty$. Since the space average of $\psi_i(\tilde{\ix}_i)$ vanishes for every $\tilde{\ix}_i$, Theorem~\ref{uxpos} yields that for $k \geq 1$,
$$
\E\ \max_{\ix \in \T^d} |\psi_i(\ix)|^k \leq \E \Big( \max_{\tilde{\ix}_i \in \T^{d-1}} \int_{\T^1}{(\psi_{ii}(\tilde{\ix}_i))^{+} dx_i} \Big)^k \overset{k}{\lesssim} 1,\ 1 \leq i \leq d.
$$
\end{proof}

\begin{cor} \label{preW11cor}
For $k \geq 1$,
$$
\E \max_{s \in [t,t+1],\ 1 \leq i \leq d}\ \max_{\tilde{\ix}_i \in \T^{d-1}} \left|\psi_{ii}(s,\tilde{\ix}_i)\right|_1^k \overset{k}{\lesssim} 1,\quad t \geq 1.
$$
\end{cor}

\begin{proof}
For every $s$, $i$ and $\tilde{\ix}_i$, the space average of $\psi_{ii}(s,\tilde{\ix}_i)$ vanishes identically. Thus, we get
$$
\int_{\T^1}{\left| \psi_{ii}(s,\tilde{\ix}_i)\right| dx_i}=2 \int_{\T^1}{(\psi_{ii}(s,\tilde{\ix}_i))^{+} dx_i} \leq 2 \max_{\ix \in \T^d} \psi_{ii}(s,\ix),
$$
and then we apply Theorem~\ref{uxpos}.
\end{proof}
\indent 
For any vector $\ka \in \Z^d$, the Gram--Schmidt method allows us to build an orthogonal basis of $\R^d$ of the form
$$
(\ka^1=\ka,\dots,\ka^d)
$$
such that all the vectors $\ka^i$ belong to $\Z^d$. Then we can consider functions on $\T^d$ as (a subset of the set of) functions on the quotient
$$
\R^d/\Z(\ka^1,\dots,\ka^d).
$$
The corresponding coordinates will be denoted by $(y_1,\dots,y_d)$. By analogy with the notation $\tilde{\ix}_i$, we define the notation $\tilde{\y}_i$.
Consequently, when considering functions which can be written as $f(\tilde{\y}_i)$ in the new coordinates, we can apply (GN) with $d'=1$.

\begin{lemm} \label{uxposbis}
Fix a vector $\ka \in \R^d$ with integer coordinates. Denote by $X_t$ the r.v.
$$
X_{t}(\ka)=\max_{s \in [t,t+1]} \max_{\ix \in \T^d}{(\ka \cdot \nabla)^2 \psi (s,\ix)}.
$$
For $k \geq 1$, we have
$$
\E X_{t}^{k} \overset{k}{\lesssim} 1,\quad t \geq 1.
$$
\end{lemm}

\begin{proof}
The proof is exactly the same as for Theorem~\ref{uxpos}, up to a change of coordinates: it suffices to work in the orthogonal basis $(\ka^1=\ka,\dots,\ka^d)$ with the corresponding coordinates $(y_1,\dots,y_d)$.
\end{proof}

\begin{cor} \label{W11}
Fix a vector $\ka \in \R^d$ with integer coordinates. For $k \geq 1$, we have
$$
\E \max_{s \in [t,t+1],\ \tilde{\y}_1 \in \T^{d-1}} \left|\frac{\partial^2 \psi(s,\tilde{\y}_1)}{\partial y_1^2}\right|^k_{1} \overset{k}{\lesssim} 1,\quad t \geq 1.
$$
\end{cor}

\begin{proof}
This result follows from Lemma~\ref{uxposbis} in the same way as Corollary~\ref{preW11cor} follows from Theorem~\ref{uxpos}.

\begin{cor} \label{W11norm}
For $k \geq 1$,
$$
\E \max_{s \in [t,t+1]} |\ue(s)|^k_{1,1} \overset{k}{\lesssim} 1,\quad t \geq 1.
$$
\end{cor}

\begin{proof}
We have
$$
|\ue|_{1,1} \sim \sum_{1 \leq i \leq d}{|\psi_{ii}|_1}+\sum_{1 \leq i < j \leq d}{|\psi_{ij}|_1}.
$$
For every $i$, the estimate for $|\psi_{ii}|_1$ follows from Corollary~\ref{preW11cor}. Now it remains to estimate the terms corresponding to $\psi_{ij},\ i <j$. We observe that by the triangle inequality,
\begin{align} \nonumber
\int_{\T^d}{\left| \psi_{ij}(s)\right|} \leq & \frac{1}{2} \int_{\T^d}{\left| \psi_{ii}(s)\right|}+\frac{1}{2} \int_{\T^d}{\left| \psi_{jj}(s)\right|}
\\ \nonumber
&+ \frac{1}{2} \int_{\T^d}{\left| \psi_{ii}(s)+\psi_{jj}(s)-2\psi_{ij}(s)\right|},
\end{align}
and then we use Corollary~\ref{W11} for $\ka=\e_i-\e_j$.
\end{proof}
\indent
Now we recall a standard estimate of the nonlinearity $\left\langle L^m v, B(v)\right\rangle$ (see Section~\ref{prel} for the definitions of $L$ and $B$). The proof is word-to-word the same as the 1d proof in \cite{BorW}.

\begin{lemm} \label{lmubuinfty}
For $\ve \in C^{\infty}$ such that $\left|\ve\right|_{\infty} \leq N$, we have
$$
N_m(\ve)=\left| \left\langle L^m \ve, B(\ve) \right\rangle\right| \leq C' \left\|\ve\right\|_m \left\|\ve\right\|_{m+1},\quad m \geq 1,
$$
with
\begin{equation} \label{polyestimate}
C' = C_m (1+N)^{n'},
\end{equation}
where $C_m$, as well as the natural number $n'$, depend only on $m$.
\end{lemm}

\indent
Now we define a norm which is equivalent to the norm given by (\ref{Wmp}) and is adapted to the use of (GN) in the 1d setting. The idea is to replace derivatives along the multi-indices in (\ref{Wmp}) (where we differentiate repeatedly over \textit{different} directions) with a sum of derivatives along a well-chosen set of directions (i.e. in each term of the sum we differentiate repeatedly along \textit{the same} direction).  For this, we need the following result:

\begin{lemm} \label{linalg}
For every $m,d \geq 1$, there exists a finite set $\Pi_m^d$ of homogeneous polynomials of degree 1 in $d$ variables $X_1,\dots,X_d$ with integer coefficients such that their $m$-th powers form a basis for the real vector space of homogeneous polynomials of degree $m$ in $d$ variables.
\end{lemm}
For a proof of this result, see Appendix 2. We will always use a fixed set $\Pi_m^d$ for fixed values of $m,d$. Moreover, every time that we consider an element $\ka \in \Pi_m^d$, we will always take a fixed \enquote{canonical} orthogonal basis $(\ka^1=\ka,\dots,\ka^d)$ as above.
\\ \indent
Now we define a norm equivalent to $|\cdot|_{m,p}$.

\begin{cor} \label{linalgcor}
For $m \geq 0,\ p \in [1,+\infty)$, the following quantity is equivalent to the norm $|\cdot|_{m,p}$:
\begin{align} \label{preSobalt}
|\ve |^{\sim}_{m,p} &=\sum_{\ka \in \Pi_m^d} \Bigg( \int_{\y \in \T^{d}}{\Bigg|\frac{\partial^m \ve}{\partial y_1^m} \Bigg|^p}\Bigg)^{1/p}
\\ \label{Sobalt}
&=\sum_{\ka \in \Pi_m^d} \Bigg( \int_{\tilde{\y}_1 \in \T^{d-1}} \Bigg( \int_{y_1 \in \T^1}{\Bigg|\frac{\partial^m \ve(\tilde{\y}_1)}{\partial y_1^m} \Bigg|^p dy_1} \Bigg) d \tilde{\y}_1 \Bigg)^{1/p},
\end{align}
where we canonically identify $\Z^d$ with the set of homogeneous polynomials 
of degree 1 in $d$ variables with integer coefficients by the correspondence
$$
(n_1,\dots,n_d) \mapsto \sum_{1 \leq i \leq d}{n_i X_i}.
$$
\end{cor}

By analogy with the notation introduced above, we denote $| \ve |^{\sim}_{m,2}$ by $\Vert \ve \Vert^{\sim}_{m}$.
\\ \indent
Note that no analogous characterisation of the $W^{m,\infty}$ norms as an average of 1d norms exists: we would have to consider the essential upper value of these norms instead.
\indent \medskip \\
\textbf{Proof of Corollary~\ref{linalgcor}.}  We note that by H{\"o}lder's inequality each of the quantities raised to the power $p$ in (\ref{preSobalt}) can be bounded from above by a linear combination of the summands in (\ref{Wmp}) raised to the power $p$. Using H{\"o}lder's inequality again, we obtain that $| \ve |^{\sim}_{m,p} \lesssim | \ve |_{m,p}$.
\\ \indent
To prove the converse inequality, we identify the set of linear constant coefficient  differential operators with the set of polynomials in $d$ variables:
$$
\sum_{\boldsymbol\alpha \in A}{c_{\boldsymbol\alpha}\ \frac{\partial^{|\boldsymbol\alpha|}}{\partial X_1^{\alpha_1} \dots \partial X_d^{\alpha_d}} } \mapsto \sum_{\boldsymbol\alpha \in A}{c_{\boldsymbol\alpha} X_1^{\alpha_1} \dots  X_d^{\alpha_d}}.
$$
This allows us to see that by  Lemma~\ref{linalg}, a differential operator $\partial^{|\boldsymbol\alpha|}/\partial x^{\boldsymbol\alpha}$ can be written as a linear combination of the operators
$$
(\partial/\partial \tilde{\y}_1)^{|\boldsymbol\alpha|}.
$$
corresponding to the vectors $\ka \in \Pi_m^d$. Now the converse inequality follows by H{\"o}lder's inequality in the same way as the direct one.
\end{proof}  
\indent
The following upper estimate of $\E \left\|\ue(t)\right\|^{2}_m$ holds uniformly for $t \geq 2$.
The proof is very similar to the proof in 1d. The only delicate point is that to get the right power of $\nu$, we have to use (GN) in a 1d setting, which requires the use of the $| \cdot |^{ \sim }$ norms.

\begin{lemm} \label{uppermaux}
For $m \geq 1$,
$$
\E \left\|\ue(t)\right\|^{2}_m \overset{m}{\lesssim} \nu^{-(2m-1)},\quad t \geq 2.
$$
\end{lemm}

\begin{proof} Fix $m \geq 1$. We will use the notation
\begin{equation} \nonumber
x(s)=\E \left\|\ue(s)\right\|^2_m;\quad y(s)=\E \left\|\ue(s)\right\|^2_{m+1}.
\end{equation}
It suffices to consider the case $t=2$. We claim that for $s \in [1,2]$ we have the implication
\begin{align} \nonumber
x(s) &\geq C' \nu^{-(2m-1)} \Longrightarrow
\\ \label{decrm}
\frac{d}{ds} x(s) &\leq -(2m-1) (x(s))^{2m/(2m-1)},
\end{align}
where $C' \geq 1$ is a fixed number, chosen later. Below, all constants denoted by $C$ are positive and do not depend on $C'$ and we denote by $Z$ the quantity
$$
Z=C' \nu^{-(2m-1)}.
$$
Indeed, assume that $x(s) \geq Z.$ By (\ref{Itoexpdiff}) and Lemma~\ref{lmubuinfty}, we have
\begin{align} \nonumber
\frac{d}{ds} x(s) \leq & - 2 \nu \E \left\|\ue(s)\right\|_{m+1}^2 + C \E \Big( (1+\left|\ue(s)\right|_{\infty})^{n'} \left\|\ue(s)\right\|_m 
\\ \label{decrmbis}
&\times \left\|\ue(s)\right\|_{m+1} \Big)+I_{m+1},
\end{align}
with $n'=n'(m)$. 
\\ \indent
Now we denote by $\varphi$ the zero mean value function
$$
\varphi=\psi-\int_{\T^d}{\psi}
$$
For each $\ka \in \Pi_{m+1}^d$, we apply (GN) to the corresponding 1d restrictions
$$
\frac{\partial^2 \varphi}{\partial y_1^2}(s,\tilde{\y}_1)=\frac{\partial^2 \psi}{\partial y_1^2}(s,\tilde{\y}_1),\ \tilde{\y}_1 \in \T^{d-1}.
$$
We get
\begin{align} \nonumber
\Vert  \varphi(s,\tilde{\y}_1) \Vert^2_{m+1} \leq C &
\Vert \varphi(s,\tilde{\y}_1) \Vert_{m+2}^{(4m-2)/(2m+1)} | \varphi(s,\tilde{\y}_1) |_{2,1}^{4/(2m+1)}.
\end{align}
Integrating in $\tilde{\y}_1$ and then summing over all $\ka \in \Pi_{m+1}^d$, we get
\begin{align} \nonumber
&\sum_{\ka \in \Pi_{m+1}^d} \int_{\tilde{\y}_1 \in \T^{d-1}}{ \Vert \varphi(s,\tilde{\y}_1) \Vert^2_{m+1}  d \tilde{\y}_1 }
\\ \nonumber
&\lesssim \sum_{\ka \in \Pi_{m+1}^d} \int_{\tilde{\y}_1 \in \T^{d-1}}{ \Vert \varphi(s,\tilde{\y}_1) \Vert_{m+2}^{(4m-2)/(2m+1)} | \varphi(s,\tilde{\y}_1) |_{2,1}^{4/(2m+1)} d \tilde{\y}_1 }.
\end{align}
By Corollary~\ref{linalgcor} and H{\"o}lder's inequality, we obtain that
\begin{align} \nonumber
\left\| \ue(s) \right\|^2_{m} \sim (\left\| \varphi(s) \right\|^{\sim}_{m+1} )^2
& \leq C (\left\| \varphi(s) \right\|^{\sim}_{m+2})^{(4m-2)/(2m+1)} 
N_{max}^{4/(2m+1)}
\\ \label{GN11mprel}
& \leq C \left\| \ue(s) \right\|_{m+1}^{(4m-2)/(2m+1)} 
N_{max}^{4/(2m+1)},
\end{align}
where $N_{max}$ is the maximum over all $\ka \in \Pi_{m+1}^d$ and all $\tilde{\y}_1 \in \T^{d-1}$ of the quantity $|\varphi(s,\tilde{\y}_1)|_{2,1} \sim |\ue(s,\tilde{\y}_1)|_{1,1}$.
\\ \indent
Consequently, (\ref{decrmbis}) yields that
\begin{align} \nonumber
\frac{d}{ds}  x(s)  \leq & - 2 \nu \E \left\|\ue(s)\right\|_{m+1}^2 + C\E \Big( (1+|\ue(s)|_{\infty})^{n'} N_{max}^{2/(2m+1)}
\\ \nonumber
& \times \left\|\ue(s)\right\|_{m+1}^{4m/(2m+1)} \Big)+I_{m+1}.
\end{align}
Corollary~\ref{W11} tells us that all moments of $N_{max}$ are finite. Thus by H{\"o}lder's inequality and Corollary~\ref{Lpupper} we get
\begin{align} \nonumber
\frac{d}{ds} x(s) \leq & \Big( - 2 \nu (y(s))^{1/(2m+1)} + C \Big) (y(s))^{2m/(2m+1)}+I_{m+1}.
\end{align}
On the other hand, (\ref{GN11mprel}) and H{\"o}lder's inequality yield that
\begin{align} \nonumber
x(s) \leq & C ( y(s) )^{(2m-1)/(2m+1)} (\E N_{max}^2)^{2/(2m+1)}
\\ \nonumber
\leq & C (y(s))^{(2m-1)/(2m+1)},
\end{align}
and thus
\begin{align} \nonumber
( y(s) )^{1/(2m+1)}& \geq C ( x(s) )^{1/(2m-1)}.
\end{align}
Consequently, since by assumption $x(s) \geq C' \nu^{-(2m-1)}$, for $C'$ large enough we have
\begin{align} \nonumber
\frac{d}{ds} x(s) \leq &\left(-C C'^{1/(2m-1)}+C \right) ( x(s) )^{2m/(2m-1)}+I_{m+1}.
\end{align}
Thus we can choose $C'$ in such a way that (\ref{decrm}) holds.
\\ \indent
Now we claim that
\begin{equation} \label{decrmcor}
x(2) \leq Z.
\end{equation}
Indeed, if $x(s) \leq Z$ for some $s \in \left[1,2\right]$, then the assertion (\ref{decrm}) ensures that $x(s)$ remains below this threshold up to $s=2$.
\\ \indent
Now, assume that $x(s) > Z$ for all $s \in \left[1,2\right]$. Denote
$$
J(s)=(x(s))^{-1/(2m-1)},\ s \in \left[1,2\right].
$$
Using the implication (\ref{decrm}) we get $dJ(s)/ds \geq 1$. Therefore $J(2) \geq 1$. As $\nu \leq 1$ and $C' \geq 1$, we get $x(2) \leq Z$. Thus in both cases inequality (\ref{decrmcor}) holds. This proves the lemma's assertion.
\end{proof}
\indent
The following statement is proved using the 1d (GN) trick as above and then proceeding in the same way as in \cite{BorW}.

\begin{cor} \label{uppermauxcor}
For $m \geq 1$,
$$
\E ||\ue(t) ||^{k}_m \overset{m,k}{\lesssim} \nu^{-k(2m-1)/2},\quad k \geq 1,\ t \geq 2.
$$
\end{cor}

\begin{proof}
The cases $k=1,2$ follow immediately from Lemma~\ref{uppermaux}. For $k \geq 3$, we consider only the case when $k$ is odd, since the general case follows by H{\"o}lder's inequality. We set $M=(k(2m-1)+1)/2$ and as previously, for $\ka \in \Pi_{m+1}^d$ we apply (GN) to the corresponding restrictions
$$
\frac{\partial^2 \psi}{\partial y_1^2}(t,\tilde{\y}_1),\ \tilde{\y}_1 \in \T^{d-1}.
$$
We get
\begin{align} \nonumber
\Vert \ue(t,\tilde{\y}_1) \Vert_{m}^{2} \overset{m,k}{\lesssim}&
\Vert \ue(t,\tilde{\y}_1) \Vert_{M}^{2/k} | \ue(t,\tilde{\y}_1) |_{1,1}^{(2k-2)/k}.
\end{align}
Integrating in $\tilde{\y}_1$ and then summing over all $\ka \in \Pi_{m+1}^d$ and using H{\"o}lder's inequality, we obtain that
\begin{align} \nonumber
\left\| \ue(t) \right\|^{k}_{m} &\overset{k,m}{\lesssim} \left\| \ue(t) \right\|_{M} N_{max}^{k-1},
\end{align}
where $N_{max}$ is the same as in the proof of Lemma~\ref{uppermaux}. Since all moments of $N_{max}$ are finite, by H{\"o}lder's inequality and Lemma~\ref{uppermaux} we get
\begin{align} \nonumber
\E \left\| \ue(t) \right\|^{k}_{m} & \overset{k,m}{\lesssim} ( \E \left| \ue(t) \right\|_{M}^{2})^{1/2} \overset{k,m}{\lesssim} \nu^{-(2M-1)/2}=\nu^{-k(2m-1)/2}.
\end{align}
\end{proof}
\indent
The following lemma follows from the corollary in exactly the same way as in 1d (see \cite[Lemma~3.8.]{BorW})

\begin{lemm} \label{uppermlemm}
For $m \geq 1$,
$$
\E  \max_{s \in [t,t+1]} || \ue(s) ||^{2}_m \overset{m}{\lesssim} \nu^{-(2m-1)},\quad t \geq 2.
$$
\end{lemm}

Now denote $\gamma=\max (0,m-1/p)$. The following result is proved for $m \geq 1$ and $p \in (1,\infty)$ in the same way as Corollary~\ref{uppermauxcor}. For $m=0$ and $p \in [1,+\infty]$ or  $m=1$ and $p=1$, this result is just a reformulation of Corollary~\ref{Lpupper} and Corollary~\ref{W11norm}, respectively. 

\begin{theo} \label{upperwmp}
For $m=0$ and $p \in [1,+\infty]$, for $m=1$ and $p \in [1,+
\infty)$, or for $m \geq 2$ and $p \in (1,\infty)$,
$$
\Big( \E  \max_{s \in [t,t+1]} | \ue(s) |^{\alpha}_{m,p} \Big)^{1/\alpha} \overset{m,p,\alpha}{\lesssim} \nu^{-\gamma},\quad \alpha>0,\ t \geq 2.
$$
\end{theo}

\indent
The problem with the 1d (GN) trick is that it cannot be used for $p=\infty$. However, it allows us to prove a slightly weaker statement, which is sufficient to obtain sharp estimates for small-scale quantities in Section~\ref{turb}.

\begin{theo} \label{avoirlinftyupper}
For $m \geq 1$ and $1 \leq i \leq d$, we have
\begin{equation}
\Big( \E  \max_{s \in [t,t+1]} \int_{\tilde{\ix}_i} {| \psi_i(s,\tilde{\ix}_i)|_{m,\infty}^{\alpha}} \Big)^{1/\alpha} \overset{m,\alpha}{\lesssim} \nu^{-m},\ \alpha>0,\ t \geq 2.
\end{equation}
\end{theo}

\section{Lower estimates for Sobolev norms} \label{lower}

In 1d, the lower estimates all follow by (GN) from the lower estimate for the quantity
$$
\frac{1}{T} \int_{t}^{t+T}{\E \Vert u(s) \Vert_1^2 },
$$ 
which is obtained by considering the energy dissipation identity for $|u|^2$. Indeed, using It{\^o}'s formula and integrating by parts (see (\ref{Itoexpdiff})) we get
$$
\E |u(t+T)|^2-\E |u(t)|^2= -2 \nu \int_{t}^{t+T}{\E \Vert u(s) \Vert_1^2 }+2I_1 T.
$$
Since there exists $T_0$ such that for $\tau \geq T_0$ we have $\E |u(\tau)|^2 \leq C_1$, for $T \geq \max(T_0,\ C_1 I_1^{-1})$ we get
$$
\frac{1}{T} \int_{t}^{t+T}{\E \Vert u(s) \Vert_1^2 } \geq  \Big( I_1-\frac{C_1}{2T} \Big) \nu^{-1} \geq \frac{I_1}{2} \nu^{-1},
$$
and we obtain a lower estimate which is uniform in $T$ for large $T$.
\\ \indent
In the multi-d case, there is an additional trilinear term
$$
\E \int_{\T^d}{(\nabla f(\ue) \cdot \nabla) \ue \cdot \ue}
$$
in the derivative in time of $\E | \ue |^2$. To estimate from below the expected value of $\frac{1}{T} \int_{1}^{T+1}{\left\|\ue(s)\right\|_1^2}$, our strategy is to prove that with a positive probability, the integral in time of the trilinear term in the right-hand side of (\ref{Itoexpdiff}) for $m=0$ is small compared to the energy supplied by the stochastic forcing.
\\ \indent
The idea is to use a \enquote{partial It{\^o}'s formula}, i.e. It{\^o}'s formula applied only to one component of the noise. This is where we use the diagonal assumption, which tells us that the noises along the different Fourier modes are independent. It is unclear whether we can drop this assumption in the general case. 
\\ \indent
Without loss of generality, in the results below until the proof of Theorem~\ref{finitetime} included, we assume that the coefficient $a_{\e_1}$ of the noise is non-trivial, i.e. $a_{\e_1}>0$, where $\e_1=(1,0,\dots,0)$. Indeed, the proof of all the results in this section in the case $b_{\e_1}>0$ is word-to word the same. On the other hand, we can reduce the general case $a_{\ka}>0$ (resp., $b_{\ka}>0$) to these cases by working in the orthogonal basis $(\ka^1=\ka,\dots,\ka^d)$: the result of Lemma~\ref{finitetime} still holds when we pass back to the canonical basis $(\e_1,\dots,\e_d)$, since the corresponding $H^1$ norms are equivalent.
\\ \indent
Consider the disintegration of $(\Omega,\F,\Pe)$ into the probability spaces
\\
$(\Omega_1,\F_1,\Pe_1)$ and $(\Omega_2,\F_2,\Pe_2)$, corresponding respectively to the Wiener process $w_{\e_1}$ and to all the other Wiener processes $w_{\ka},\ \tilde{w}_{\ka}$. An element $\omega \in \Omega$ will be decomposed as $(\omega_1,\omega_2)$. For $\omega_1 \in \Omega_1$ (resp. $\omega_2 \in \Omega_2$), $\Pe_{\omega_1}$ (resp. $\Pe_{\omega_2}$) denotes the disintegration $\Pe(\cdot| \omega_1)$ (resp. $\Pe(\cdot| \omega_2)$). The notation $\E_1$, $\E_2$, $\E_{\omega_1}$, $\E_{\omega_2}$ is defined accordingly. 
\\ \indent
From now on and until the end of the proof of Lemma~\ref{finitetime} we fix $\tau \geq 1$ and we put
$$
w(s):=w(s)-w(\tau),
$$
and similarly for $w_{\e_1}$. The exact value of $\tau$ is unimportant, since all the estimates in this section will hold uniformly in $\tau$. Our modification for the definition of the Wiener process does not change the white noise and thus has no effect on the solutions $\ue$. It will considerably simplify the notation below.
\\ \indent
Now consider the $\omega_1$-independent difference
$$
\tilde{w}(s,\ix)=w(s,\ix)-a_{\e_1} w_{\e_1}(s,\ix) \cos(2 \pi x_1),
$$
and for $\kappa > 1$, define the event $Z=Z(\kappa)$ by
$$
Z=\left\{ \omega_2:\ \max_{t \in [\tau,\tau+2\kappa^{-1}]}{|\tilde{w^{\omega}}(t)|_{4,\infty}} \leq \kappa^{\Lambda} \right\}.
$$
The constant $\Lambda>2$ will be chosen in the proof of Theorem~\ref{uxpos2}. By the assumption $\mathbf{i)}$ in Section~\ref{rand}, for any $\kappa>0$ we have $\Pe_2(Z)>0$.
\\ \indent
The proof of the following result is similar to the proof of Theorem~\ref{uxpos}.

\begin{theo} \label{uxpos2}
For $\kappa \in (0,1)$, denote by $X_i$ the r.v.
$$
X_i=\max_{t \in [\tau+\kappa^{-1},\ \tau+2 \kappa^{-1}],\ \ix \in \T^d} \psi_{ii}(t,\ix),\ 2 \leq i \leq d.
$$
Then for every $\omega_2 \in Z(\kappa)$ and $k \geq 1$ we have
$$
\E_{\omega_2} \ \Big(\max_{2 \leq i \leq d } X_{i} \Big)^{k} \overset{k}{\lesssim} \kappa^k.
$$
\end{theo}

\begin{proof}
It suffices to prove the statement with $(t-\tau) \psi_{ii}$ in place of $\psi_{ii}$, the interval $\left[\tau,\ \tau+2 \kappa^{-1} \right]$ in place of $\left[\tau+\kappa^{-1},\ \tau+2 \kappa^{-1} \right]$ and 1 in place of $\kappa^k$. Without loss of generality, we can consider the case where the maximum of $(t-\tau) \psi_{ii}$ on $S=\left[\tau,\ \tau+2 \kappa^{-1} \right] \times \T^d$ for $i \in [2,d]$ is reached for $i=2$. Denote this maximum by $N$.
\\ \indent
Now consider the equation (\ref{HJ}). Differentiating twice in $x_2$, we get
\begin{align} \nonumber
&(\psi_{22}-w_{22})_t+\sum_{1 \leq i,j \leq d}{\psi_{2i} \psi_{2j} f_{ij}(\nabla \psi)}+\sum_{1 \leq i \leq d}{(\psi_{22})_i f_i(\nabla \psi)} =\nu \Delta \psi_{22}.
\end{align}
Putting $v=\psi-w$ and using (\ref{strconvex}), we get
\begin{align} \nonumber
&(v_{22})_t + \sigma |\nabla \psi_2|^2 + \sum_{1 \leq i \leq d}{(v_{22})_i f_i(\nabla \psi)} + \sum_{1 \leq i \leq d}{w_{22i} f_i(\nabla \psi)}
\\ \label{v22}
&\leq \nu \Delta v_{22}+\nu \Delta w_{22}.
\end{align}
Now multiply (\ref{v22}) by $(t-\tau)^2$ and consider the function $\tilde{v}=(t-\tau)v_{22}$. By assumption, the maximum of this function equals $N$. The function $\tilde{v}$ satisfies
\begin{align} \nonumber
&(t-\tau) \tilde{v}_t -\tilde{v} +\sigma (\tilde{v}+(t-\tau) w_{22})^2 + (t-\tau) \sum_{1 \leq i \leq d}{\tilde{v}_i f_i(\nabla \psi)}  
\\ \label{tildev2}
& + (t-\tau)^2 \sum_{1 \leq i \leq d}{w_{22i} f_i(\nabla \psi)} \leq \nu (t-\tau) \Delta \tilde{v}+\nu (t-\tau)^2 \Delta w_{22}.
\end{align}
If the zero mean function $\tilde{v}$ does not vanish identically on the domain $S=\left[\tau,\ \tau+2 \kappa^{-1} \right] \times \T^d$, then it attains its positive maximum $N$ on $S$ at a point $(t^0,\ix^0)$ such that $t^0>\tau$. At such a point we have $\tilde{v}_t \geq 0$, $\tilde{v}_i=0$ for all $i$ and $\Delta \tilde{v} \leq 0$. By (\ref{tildev2}), at $(t^0,\ix^0)$ we have the inequality
\begin{equation} \label{maxpoint2}
\sigma (\tilde{v}+(t-\tau)w_{22})^2 \leq \nu (t-\tau)^2 \Delta w_{22}+\tilde{v} - (t-\tau)^2 \sum_{1 \leq i \leq d}{w_{22i} f_i(\nabla \psi)}.
\end{equation}
Using (\ref{poly}) (we recall that $\delta$ is by definition $2-h(1)$) and the definition of $Z(\kappa)$, we get
\begin{align} \nonumber 
&\Big| (t-\tau)^2 \sum_{1 \leq i \leq d}{w_{22i} f_i(\nabla \psi)} \Big| \leq C \kappa^{\Lambda} (t-\tau)^2 (1+|\nabla \psi|)^{2-\delta}.
\end{align}
Now denote by $K$ the r.v.
$$
K=K(\omega_1)=\max_{t \in [\tau,\tau+2\kappa^{-1}]} |w_{\e_1}(t)|_{4,\infty}.
$$
By the same arguments as in the proof of Theorem~\ref{uxpos} (see (\ref{maxpoint})), we obtain that if
$$
\tilde{N} \geq 2 \kappa^{-1} (K+\kappa^{\Lambda}),
$$
then the maximum $\tilde{N}$ of $(t-\tau) (\psi_{ii}-w_{ii})$ over all $i$ (including $i=1$) on $S$ satisfies
\begin{align} \nonumber 
&\sigma \Big( \tilde{N}-2 \kappa^{-1} (K+\kappa^{\Lambda}) \Big)^2 
\\ \nonumber
&\lesssim \kappa^{-2} (K+\kappa^{\Lambda}) + \tilde{N} + \kappa^{-2} (K+\kappa^{\Lambda}) (1+K+\kappa^{\Lambda}+\tilde{N})^{2-\delta}.
\end{align}
Therefore, since $\kappa<1$ and $0 < \delta \leq 1$, we get
\begin{align} \label{Ntilde}
\tilde{N} \lesssim (1+K)^{1/\delta} \kappa^{-2/\delta}.
\end{align}
For every $i$ and every $\tilde{\ix}_i$, $(t-\tau) \psi_i(\tilde{\ix}_i)$ is the zero mean primitive of $(t-\tau) \psi_{ii}(\tilde{\ix}_i)$ and thus on $S$ we have
$$
(t-\tau) \psi_i \leq \tilde{N}+ 2 \kappa^{-1} (K+\kappa^{\Lambda}),\ 1 \leq i \leq d.
$$
Consequently, at $(t^0,\ix^0)$, by (\ref{poly}) and (\ref{Ntilde}), we get
\begin{align} \nonumber 
\Big| (t-\tau)^2 \sum_{1 \leq i \leq d}{w_{22i} f_i(\nabla \psi)} \Big| &\leq  C \kappa^{\Lambda} (t-\tau)^{\delta} (\kappa^{-1} \tilde{N}+ \kappa^{-2} K+\kappa^{\Lambda-2})^{2-\delta}
\\ \nonumber
& \leq C \kappa^{\Lambda-\delta} (1+K)^{2/\delta-1} \kappa^{\delta-4/\delta}
\\ \nonumber
& \leq C \kappa^{\Lambda-4/\delta} (1+K)^{(2-\delta)/\delta}.
\end{align}
From now on, we assume that $N \geq 1$. Since $\nu \in (0,1]$, the relation (\ref{maxpoint2}) yields that
$$
\sigma (N-2 \kappa^{-1} \kappa^{\Lambda})^2 \leq C \kappa^{-2} \kappa^{\Lambda} +N+C \kappa^{\Lambda-4/\delta} (1+K)^{(2-\delta)/\delta}.
$$
Now we put
$$
\Lambda=\frac{8}{ \delta};
$$
we recall that $\delta \leq 1$. Then we get
$$
N \leq C (1+K^{(2-\delta)/2 \delta}  \kappa^{2/ \delta}) \leq C (1+K^{1/\delta} \kappa^{2/\delta}).
$$
Since by (\ref{moments}) all moments of $\kappa^{1/2} K$ are bounded, for $X_2$ and thus by assumption for all $X_i,\ 2 \leq i \leq d$, we have
$$
\E_{\omega_2} X_i^k \overset{k}{\lesssim} 1,\ k \geq 0.
$$
\end{proof}
\indent
The following result follows from Theorem~\ref{uxpos2} in the same way as Corollary~\ref{Lpupper} and Corollary~\ref{preW11cor} follow from Theorem~\ref{uxpos}.

\begin{cor} \label{uxpos2cor}
Fix $\kappa>0$. Then for every $i \in [2,d]$, $\omega_2 \in Z(\kappa)$ and $k \geq 1$, we have respectively
$$
\E_{\omega_2} \max_{t \in [\tau+\kappa^{-1},\ \tau+2 \kappa^{-1}],\ \tilde{\ix}_i \in \T^d} \left|\psi_i(t,\tilde{\ix}_i)\right|^k_{p} \overset{k}{\lesssim} \kappa^{k},\quad p \in [1,\infty]
$$
and
$$
\E_{\omega_2} \max_{t \in [\tau+\kappa^{-1},\ \tau+2 \kappa^{-1}],\ \tilde{\ix}_i \in \T^d} \left|\psi_{ii}(t,\tilde{\ix}_i)\right|^k_{1} \overset{k}{\lesssim} \kappa^{k}.
$$
\end{cor}

The following result is proved in the same way as Corollary~\ref{W11}.

\begin{lemm} \label{W11omega2}
Fix $\kappa>0$. Then for every $t \in [\tau+\kappa^{-1},\tau+2 \kappa^{-1}-1]$ and $k \geq 1$, we have 
$$
\E_{\omega_2} \max_{s \in [t,t+1]} \left|\ue(t)\right|_{1,1}^k \overset{k}{\lesssim} 1.
$$
\end{lemm}

Now we are ready to prove the following crucial estimate.

\begin{lemm} \label{finitetime}
There exists a constant $\kappa \in (0,1)$ such that we have
$$
\kappa \int_{\tau+\kappa^{-1}}^{\tau+2 \kappa^{-1}}{\ \E \left\|\ue(s)\right\|_1^2} \gtrsim \nu^{-1}.
$$
\end{lemm}

\begin{proof} Since $P_2(Z(\kappa))>0$, it suffices to prove this lower estimate for $\omega_2 \in Z(\kappa)$ and $\E_{\omega_2}$ in place of $\E$. In the proof below, until the final steps we will not indicate explicitly the dependence on
$$
s \in [t+\kappa^{-1},\ t+2 \kappa^{-1}].
$$
Writing (\ref{whiteBurgers}) as a stochastic PDE with $\omega_2$ fixed, we get
$$
\frac{\partial \ve}{\partial t} +  (\nabla f(\ue) \cdot \nabla) \ue - \nu \Delta \ue = -2 \pi a_{\e_1} \frac{\partial w_{\e_1}}{\partial t} \sin(2 \pi x_1) \e_1,
$$
where $\ve$ denotes
$$
\ve=\ue-\nabla \tilde{\tilde{w}},\ \tilde{\tilde{w}}(s)=\tilde{w}(s)-\tilde{w}(t+\kappa^{-1}).
$$
Applying It{\^o}'s formula and integrating by parts, we obtain that
\begin{align} \nonumber
&\frac{1}{2} \frac{\partial (\E_{\omega_2} | \ve|^2)}{\partial t} 
\\ \nonumber
&= -\E_{\omega_2} \int_{\T^d}{\Big( (\nabla f(\ue) \cdot \nabla) \ue \cdot \ve \Big)}+\nu\ \E_{\omega_2} \langle \Delta \ue,\ \ve \rangle + \pi^2 a_{\e_1}^2
\\ \nonumber
&= -\E_{\omega_2} \sum_{i,j=1}^{d} \int_{\T^d}{f_i(\ue) \psi_{ij} v_j} + \nu \ \E_{\omega_2} \langle \ue,\ \nabla \Delta \tilde{\tilde{w}} \rangle - \nu\ \E_{\omega_2} \Vert \ue \Vert_1^2 + \pi^2 a_{\e_1}^2
\\ \label{A+B}
&= A+B,
\end{align}
where
\begin{equation} \nonumber
A=-\E_{\omega_2} \sum_{i,j=1}^{d} \int_{\T^d}{f_i(\ue) \psi_{ij} v_j} + \nu \ \E_{\omega_2} \langle \ue,\ \nabla \Delta \tilde{\tilde{w}} \rangle
\end{equation}
and
\begin{equation} \nonumber
B=-\nu\ \E_{\omega_2} \Vert \ue \Vert_1^2 + \pi^2 a_{\e_1}^2.
\end{equation}
The sum $B$ is similar to the expression in 1d, since it is the sum of a dissipative term due to the Laplacian and of a pumping term due to the forcing. Since we want to use the same mechanism as in 1d, our goal is to prove that $A$ is small. We have
\begin{align} \nonumber
|A| \leq& \Big| \E_{\omega_2} \sum_{i,j=1}^{d} \int_{\T^d}{f_i(\ue) \psi_{ij} \psi_j} \Big| + \Big| \E_{\omega_2} \sum_{i,j=1}^{d} \int_{\T^d}{f_i(\ue) \psi_{ij} \tilde{\tilde{w}}_j} \Big|
\\ \nonumber
& + \nu \ \Big| \E_{\omega_2} \langle \ue,\ \nabla \Delta \tilde{\tilde{w}} \rangle \Big|.
\end{align}
By the definition of $Z(\kappa)$ we have $|\tilde{\tilde{w}}|_{4,\infty} \leq \kappa^{\Lambda}$. Consequently, using Lemma~\ref{W11omega2} and (\ref{poly}), we obtain that the second and the third terms in the right-hand side are uniformly bounded from above by $C \kappa^{\Lambda}$ and $C \nu \kappa^{\Lambda}$, respectively. Thus, integrating by parts and then using Theorem~\ref{uxpos2}, Lemma~\ref{W11omega2} and (\ref{poly}) we obtain that
\begin{align} \nonumber
|A| \leq& \Big| \E_{\omega_2} \int_{\T^d}{f_1(\ue) \psi_1 \psi_{11} } \Big|
\\ \nonumber
&+ \frac{1}{2} \Big| \E_{\omega_2} \sum_{i,j \in [1,d],\ (i,j) \neq (1,1)} \int_{\T^d}{f_{ii}(\ue) \psi_{ii} \psi_j^2} \Big|+C \kappa^{\Lambda}
\\ \label{A2}
\leq& \Big| \E_{\omega_2} \int_{\T^d}{f_1(\ue) \psi_1 \psi_{11} } \Big|+C \kappa+C \kappa^{\Lambda}.
\end{align}
To prove that the first term in the right-hand side of (\ref{A2}) is small, we consider the function $g$ defined by
$$
g(\ix)=f_1(\psi_1(\ix),\psi_2(0,\tilde{\ix}_1),\dots,\psi_d(0,\tilde{\ix}_1)).
$$
We get
\begin{align} \nonumber
 & \Big|  \E_{\omega_2}  \int_{\T^d}{f_1(\ue) \psi_{1} \psi_{11}} \Big|
\\ \nonumber
\leq & \Big| \E_{\omega_2} \int_{\T^d}{g(\ue) \psi_1 \psi_{11}} \Big|
\\ \nonumber
&  + \Big| \E_{\omega_2} \int_{\T^d}{ \Big( f_1((\psi_1(\ix),\psi_2(0,\tilde{\ix}_1),\dots,\psi_d(0,\tilde{\ix}_1)))} 
\\ \nonumber
&  {-f_1((\psi_1(\ix),\psi_2(\ix),\dots,\psi_d(\ix))) \Big) \psi_1 \psi_{11}} \Big|
\\ \label{IPP}
\leq &  \Big| \E_{\omega_2} \int_{\T^{d-1}}{\Big( \int_{\T^1}{g(\ue(\tilde{\ix}_1)) \psi_1(\tilde{\ix}_1) \psi_{11}(\tilde{\ix}_1)} dx_1 \Big) d \tilde{\ix}_1} \Big|
\\ \nonumber
& + \E_{\omega_2} \Bigg(  |\psi_{1}|_{\infty} |\psi_{11}|_1 
\\ \nonumber
&   \times 2V (d-1) \max_{\ix \in \T^d,\ a \in [0,1],\ 2 \leq i \leq d}{\Bigg| f_{1i} (\psi_1(\ix),\psi_2(ax_1,\tilde{\ix}_1),\dots,\psi_d(ax_1,\tilde{\ix}_1)) \Bigg|} \Bigg).
\end{align}
Here, $V$ denotes
$$
\sum_{2 \leq i \leq d}{\max_{\ix \in \T^d}{|\psi_i(\ix)|}}.
$$
The first term in (\ref{IPP}) is equal to $0$. Indeed, the integrand is a full derivative in $x_1$, for every $\tilde{\ix}_1$. Then by H{\"o}lder's inequality and (\ref{poly}) we get
\begin{align} \nonumber
& \Big| \E_{\omega_2} \int_{\T^d}{f_1(\ue) \psi_{1} \psi_{11}} \Big|
\\ \nonumber
& \leq C (\E_{\omega_2} V^2)^{1/2} \Big(\E_{\omega_2} ((1+|\ue|_{1,1})^{2h(1)+2} |\ue|_{\infty}^2) \Big)^{1/2}.
\end{align}
By Corollary~\ref{uxpos2cor} and Lemma~\ref{W11omega2}, this quantity is bounded from above by $C \kappa$. Adding up the terms in (\ref{A+B}) and (\ref{A2}), we get
\begin{align} \label{dissdiff}
&\frac{1}{2} \frac{\partial (\E_{\omega_2} | \ve|_2^2)}{\partial t}  = A+B \geq -C \kappa-\nu\ \E_{\omega_2} \Vert \ue \Vert_1^2 + \pi^2 a_{\e_1}^2.
\end{align}
Now integrate (\ref{dissdiff}) in time over $[\tau+\kappa^{-1},\ \tau+2\kappa^{-1}]$. We get
\begin{align} \nonumber
& \frac{1}{\kappa^{-1}} \int_{\tau+\kappa^{-1}}^{\tau+2\kappa^{-1}}{\E_{\omega_2} \left\|\ue(s)\right\|_1^2} 
\\ \nonumber
& \geq (2\nu)^{-1} \Big(2 \pi^2 a_{\e_1}^2-2C \kappa+ \frac{\E_{\omega_2}{\left| \ve(\tau+\kappa^{-1})\right|^2}-\E_{\omega_2}{\left| \ve(\tau+2\kappa^{-1})\right|^2}}{\kappa^{-1}} \Big)
\\ \nonumber
& \geq (2\nu)^{-1} \Big(2 \pi^2 a_{\e_1}^2-2C \kappa- \frac{\E_{\omega_2}{\left| \ue(\tau+2\kappa^{-1})-\nabla \tilde{\tilde{w}}(\tau+2\kappa^{-1})\right|^2}}{\kappa^{-1}} \Big).
\end{align}
By Corollary~\ref{uxpos2cor}, there exists a constant $C'>0$ such that we have
$$
\E_{\omega_2} \left| \ue(\tau+2 \kappa^{-1})\right|^2 \leq C'.
$$
On the other hand, since $\omega_2 \in Z(\kappa)$, we have $\left| \nabla \tilde{\tilde{w}}(\tau+2\kappa^{-1}) \right|^2 \leq C \kappa^{2 \Lambda}$. Thus we get
\begin{align} \nonumber
& \frac{1}{\kappa^{-1}} \int_{\tau+\kappa^{-1}}^{\tau+2\kappa^{-1}}{\E_{\omega_2} \left\|\ue(s)\right\|_1^2} 
\\ \nonumber
& \geq (2\nu)^{-1} \Big(2 \pi^2 a_{\e_1}^2-2 C \kappa- \frac{C'-C\kappa^{\Lambda}}{\kappa^{-1}} \Big).
\end{align}
Now it remains to choose $\kappa$ small enough to prove the lemma's statement.
\end{proof}
\indent
From now on, we drop the assumption $a_{\e_1}>0$. As observed above, Lemma~\ref{finitetime} still holds without this assumption.

\begin{cor} \label{finitetimecor}
There exists a constant $\kappa \in (0,1)$ and $i \in [1,d]$ such that we have
$$
\kappa \int_{\tau+\kappa^{-1}}^{\tau+2 \kappa^{-1}}{\ \E u_{ii}^2(s)} \gtrsim \nu^{-1}.
$$
\end{cor}

\begin{proof}
By definition of the $H^1$ norm, it suffices to prove that we have
$$
|u_{ij}| \leq |u_{ii}|^{1/2} |u_{jj}|^{1/2},\ i,\ j \in [1,d],\ i \neq j.
$$
This fact is proved integrating by parts and using the Cauchy-Schwarz inequality:
\begin{align} \nonumber
|u_{ij}|^2=\int{u_{ij}^2} & = -\int{u_{iij} v_j} =\int{u_{ii} u_{jj}} \leq |u_{ii}| |u_{jj}|.
\end{align}
\end{proof}

\indent
From now on, we denote
$$
T_0=\kappa^{-1}.
$$
To generalise the lower estimate proved above to averages over time intervals of length $T \geq T_0$, it suffices to use the Markov property. The time-averaged lower bound for the $H^1$ norm obtained above yields similar bounds for $H^m$ norms with $m \geq 2$. This is done almost in the same way as in 1d. The only additional difficulty is that we apply  (GN) to 1d restrictions of $\ue$: we proceed in the same way as for the upper estimates, using the 1d (GN) trick.

\begin{lemm} \label{finalexp}
For $m \geq 1$,
$$
\frac{1}{T} \int_{t}^{t+T}{\E \left\| \ue(s)\right\|_m^2} \overset{m}{\gtrsim} \nu^{-(2m-1)},\qquad t,T \geq T_0.
$$
\end{lemm} 

\begin{proof}
In the proof below, until the final steps we will not indicate explicitly the dependence on $s \in [t,t+T]$.
\\ \indent
Since the case $m=1$ has been treated in the previous lemma, we may assume that $m \geq 2$. By (GN), for the 1d restrictions $\ue(\tilde{\ix}_i)$ we have
$$
\Vert \ue(\tilde{\ix}_i) \Vert_1^2 \lesssim \Vert\ue(\tilde{\ix}_i) \Vert_m^{2/(2m-1)} \left| \ue(\tilde{\ix}_i) \right|_{1,1}^{(4m-4)/(2m-1)}.
$$
Thus, in the same way as in the proof of Lemma~\ref{uppermaux}, using H{\"o}lder's inequality and Lemma~\ref{uxposbis} we get
\begin{align} \label{Sobolev11}
( \E \left\|\ue\right\|^2_1 )^{2m-1} &\overset{m}{\lesssim} \E \left\|\ue\right\|_m^2.
\end{align}
Integrating (\ref{Sobolev11}) in time, we get
\begin{align} \nonumber
\frac{1}{T} \int_{t}^{t+T}{\E \left\|\ue\right\|_m^2 } \overset{m}{\gtrsim} & \ \frac{1}{T} \int_{t}^{t+T}{(\E \left\|\ue\right\|^2_1)^{2m-1}}
\\ \nonumber
\overset{m}{\gtrsim} & \ \Big(\frac{1}{T} \int_{t}^{t+T}{\E \left\|\ue\right\|_1^{2} } \Big)^{2m-1}.
\end{align}
Now the lemma's assertion follows from Lemma~\ref{finitetime}.
\end{proof}
\indent
The following results generalise Corollary~\ref{finitetimecor} and Lemma~\ref{finalexp}: they are proved using the 1d (GN) trick: in the same way as in the previous section. We recall that $\gamma=\max(0,m-1/p)$. Note that Theorem~\ref{avoirlinftylower} does not necessarily hold for all $i$. For instance, consider the case when the initial condition $\psi^0$ and the noise $w$ only depend on one coordinate. Lemma~\ref{lowerLp} is the only lower estimate in our paper which holds without averaging in time.

\begin{lemm} \label{finalexpter}
For $m \geq 1$ and $p \in [1,\infty]$,
$$
\Big( \frac{1}{T} \int_{t}^{t+T}{\E \left|\ue(s)\right|_{m,p}^{\alpha}} \Big)^{1/\alpha} \overset{m,p,\alpha}{\gtrsim} \nu^{-\gamma},\quad \alpha>0,\ t,T \geq T_0.
$$
\end{lemm}

\begin{theo} \label{avoirlinftylower}
There exists $1 \leq i \leq d$ such that we have
\begin{equation}
\Big( \frac{1}{T} \int_{t}^{t+T}{ \int_{\tilde{\ix}_i} \E | \psi_i(s,\tilde{\ix}_i)|_{m,\infty}^{\alpha}} \Big)^{1/\alpha} \overset{m,\alpha}{\gtrsim} \nu^{-m},\ m \geq 0,\ \alpha>0,\ t,T \geq T_0.
\end{equation}
\end{theo}

\begin{lemm} \label{lowerLp}
For $m=0$ and $p \in [1,\infty]$, or for $m,p=1$, we have
$$
\E \left|\ue(t)\right|^{\alpha}_{m,p}\overset{\alpha}{\gtrsim} 1,\quad t \geq 2T_0,\quad \alpha>0.
$$
\end{lemm}

\section{Sobolev norms: main theorem} \label{main}

The following two theorems sum up the main results of Sections~\ref{upper}-\ref{lower}, with the exception of the upper estimates on the second directional derivatives of $\psi$ in Section~\ref{upper}. We recall that $\gamma=\max(0,m-1/p)$.

\begin{theo} \label{avoir}
For $m=0$ and $p \in [1,\infty]$, $m=1$ and $p \in [1,\infty)$, or $m \geq 2$ and $p \in (1,\infty)$,
\begin{equation} \label{asymp}
\Big( \frac{1}{T} \int_{t}^{t+T}{\E \left|\ue(s)\right|_{m,p}^{\alpha}} \Big)^{1/\alpha} \overset{m,p,\alpha}{\sim} \nu^{-\gamma},\quad \alpha>0,\ t,T \geq T_0.
\end{equation}
Moreover, the upper estimates hold with time-averaging replaced by maximising over $[t,t+1]$ for $t \geq 2$, i.e.
$$
\Big( \E \max_{s \in [t,t+1]}{\left|\ue(s)\right|_{m,p}^{\alpha}}\Big)^{1/\alpha} \overset{m,p,\alpha}{\lesssim} \nu^{-\gamma},\quad \alpha>0,\ t \geq 2.
$$
The lower estimates still hold for $m=1$ and $p=\infty$, or for $m \geq 2$ and $p=1,\infty$. On the other hand, they hold without averaging in time for $m=0$ and $p \in [1,\infty]$ or $m,p=1$, provided $t \geq 2T_0$. In other words, in this case (corresponding to $\gamma=0$) we have
\begin{equation} \nonumber
\Big( \E \left|\ue(t)\right|_{m,p}^{\alpha} \Big)^{1/\alpha} \overset{\alpha}{\sim} 1,\quad \alpha>0.
\end{equation}

\end{theo}

\begin{theo} \label{avoirlinfty}
There exists $1 \leq i \leq d$ such that for $m \geq 1$, we have
\begin{equation}
\Big( \frac{1}{T} \int_{t}^{t+T}{ \int_{\tilde{\ix}_i} \E | \psi_i(s,\tilde{\ix}_i)|_{m,\infty}^{\alpha}} \Big)^{1/\alpha} \overset{m}{\sim} \nu^{-m},\ \alpha>0,\ t,T \geq T_0.
\end{equation}
The corresponding upper estimate holds with time-averaging replaced by maximising over $[t,t+1]$ for $t \geq 2$, and for all $i$.
\end{theo}

We recall that these two theorems still hold if we replace the Sobolev norms with their suprema over all smooth initial conditions.
\\ \indent
Theorem~\ref{avoir} yields, for integers $m \geq 1$, the relation
\begin{equation} \label{integers}
\lbrace \Vert \ue \Vert_m^2 \rbrace \overset{m}{\sim} \nu^{-(2m-1)}.
\end{equation}
By a standard interpolation argument (see (\ref{Sobolevspectr})) the upper bound in (\ref{integers}) also holds for non-integer numbers $s >1$. Actually, the same is true for the lower bound, since for any integer $n>s$ we have
\begin{align} \nonumber
\lbrace \Vert \ue \Vert_s^2 \rbrace &\geq \lbrace \Vert \ue \Vert_n^2 \rbrace^{n-s+1} \lbrace \Vert \ue \Vert_{n+1}^2 \rbrace^{-(n-s)} \overset{s}{\gtrsim} \nu^{-(2s-1)}.
\end{align}

\section{Estimates for small-scale quantities} \label{turb}

In this section, we estimate the small-scale quantities which characterise Burgulence in the physical space (increments) as well as in the Fourier space (energy spectrum). The notation is given in Section~\ref{agreeturb}. Note that in this section, we use the results in Sections~\ref{upper}-\ref{main} as a "black box". In other words, we do not directly use the fact that $\ue$ solves (\ref{whiteBurgers}).

\subsection{Results in physical space} \label{phys}

We begin by proving the upper estimates for the structure functions. 

\begin{lemm} \label{upperdiss}
For $|\er|=\ell$, $1 \leq i \leq d$,
\begin{equation} \label{upperdissstat}
S_{p,\alpha,i}(\er) \overset{p,\alpha}{\lesssim} \left\lbrace \begin{aligned} & \ell^{\alpha p},\ 0 \leq p \leq 1. \\ & \ell^{\alpha p} \nu^{-\alpha(p-1)},\ p \geq 1. \end{aligned} \right.
\end{equation}
\end{lemm}

\begin{proof}
The case $p < 1$ follows from the case $p=1$ by H{\"o}lder's inequality: thus it suffices to consider $p \geq 1$. For simplicity, we only consider the case $\er=\ell \e_j$, which yields the general case by the triangle inequality. We observe that we have
\begin{align} \nonumber
S_{p,\alpha,i}(\ell \e_j) &= \Big\{ \Big( \int_{\ix \in \T^d}{|\psi_i(\ix+\ell \e_j)-\psi_i(\ix) |^p d \ix} \Big)^{\alpha} \Big\}
\\ \nonumber
& \overset{p,\alpha}{\lesssim} \Big\{ \Big( \int_{\ix \in \T^d}{ \Big| \int_{0}^{\ell}{\psi_{ij}(\ix+ y \e_j) dy} \Big|^p d \ix } \Big)^{\alpha} \Big\}.
\end{align}

Then, H{\"o}lder's inequality yields that

\begin{align} \nonumber
S_{p,\alpha,i}(\ell \e_j)& \overset{p,\alpha}{\lesssim} \Big\{ \Big( \ell^{p-1} \int_{\ix \in \T^d}{\Bigg( \int_{0}^{\ell}{|\psi_{ij}(x+y\e_j)|^p dy \Bigg) d \ix} } \Big)^{\alpha} \Big\}
\\ \nonumber
& = \ell^{\alpha p} \Big\{ |\psi_{ij}|_{p}^{\alpha p} \Big\} \overset{p,\alpha}{\lesssim} \ell^{\alpha p} \nu^{-\alpha(p-1)},
\end{align}
where the last step follows from Theorem~\ref{avoir}.
\end{proof}
\indent
The following upper bound gives a better estimate for $\ell \in J_2 \cup J_3$.

\begin{lemm} \label{upperinert}
For  $|\er|=\ell$, $1 \leq i \leq d$,
\begin{equation} \label{upperinertstat}
S_{p,\alpha,i}(\er) \overset{p,\alpha}{\lesssim} \left\lbrace \begin{aligned} & \ell^{\alpha p},\ 0 \leq p \leq 1. \\ & \ell^{\alpha},\ p \geq 1. \end{aligned} \right.
\end{equation}
\end{lemm}

\begin{proof}
As above, it suffices to consider the case $p \geq 1,\ \er=\ell \e_j$. We get
\begin{align} \nonumber
&S_{p,\alpha,i}(\ell \e_j) = \Big\{ \Big( \int_{\ix \in \T^d} {|\psi_i(\ix+\ell \e_j)-\psi_i(\ix)|^p d \ix} \Big)^{\alpha} \Big\}
\\ \nonumber
& \leq \Big\{ \Big( (2 |\psi_i|_{\infty})^{p-1} \int_{\ix \in \T^d} {|\psi_i(\ix+\ell \e_j)-\psi_i(\ix)| d \ix} \Big)^{\alpha} \Big\}
\\ \nonumber
& \overset{p,\alpha}{\lesssim} \Big\{ |\psi_i|^{2 \alpha (p-1)}_{\infty} \Big\}^{1/2} \Big\{ \Big( \int_{\ix \in \T^d} {|\psi_i(\ix+\ell \e_j)-\psi_i(\ix)| d \ix} \Big)^{2 \alpha} \Big\}^{1/2}
\\ \label{incrij}
& \overset{p,\alpha}{\lesssim} \Big\{ \Big( \int_{\ix \in \T^d} {|\psi_i(\ix+\ell \e_j)-\psi_i(\ix)| d \ix} \Big)^{2 \alpha} \Big\}^{1/2}
\\ \label{incrij2}
& = \Big( S_{1,2 \alpha,i}(\ell \e_j) \Big)^{1/2} \overset{\alpha}{\lesssim} \ell^{\alpha},
\end{align}
where (\ref{incrij}) follows from Theorem~\ref{avoir}, and (\ref{incrij2}) follows from Lemma~\ref{upperdiss}.
\end{proof}

\begin{rmq}
Note that the upper estimates in the formulation of Lemma~\ref{upperinert} hold without averaging in time in the definition of $S_{p,\alpha,i}$, i.e. for the quantities
\begin{align} \nonumber
\E \Big( \int_{\ix \in \T^d} {|\psi_i(t,\ix+\er)-\psi_i(t,\ix)|^p d \ix} \Big)^{\alpha},\quad t \geq 2.
\end{align}
Moreover, we may replace averaging in time with maximising over
\\
$[t,t+1]$.
\end{rmq}

\indent
Now we prove lower estimates for \textit{longitudinal} increments. Generally speaking, we do not have any lower estimates for the \textit{transverse increments} such as for instance $S_{p,\alpha,i}(\ell \e_j),\ i \neq j$. Indeed, these quantities vanish identically if we have an initial condition and noise which only depend on one coordinate.
\\ \indent
The main idea of the proofs below is that for some $i$, the 1d restrictions of $\psi_i$ with fixed $\tilde{\ix}_i$ exhibit, in average, the same behaviour for Sobolev norms as the 1d solutions. The value of $i$ is the same as in Theorem~\ref{avoirlinfty}: without loss of generality, we may suppose that $i=1$.
\\ \indent
Loosely speaking, the following lemma states that with a probability which is not too small, for an amount of time which is not too small and for $\tilde{\ix}_1$ in a subset of $\T^1$ which is not too small, several Sobolev norms of the corresponding 1d restrictions $\tilde{\psi}_1:=\psi_1(\tilde{\ix}_1)$ are of the same order as their average values.
\\ \indent
The only difference with the 1d case is that in addition to taking the expected value and averaging in time, we have to average the norms for restrictions $\tilde{\psi}_1$ over $\tilde{\ix}_1$: thus, the meaning of the term "in average" is different. Eventhough we do not have good upper estimates for the norms $|\ue|_{m,\infty}$, we can nevertheless estimate the small-scale quantities in exactly the same way as in 1d using Theorem~\ref{avoirlinfty}. Thus, the proofs for Lemma~\ref{typical}, Corollary~\ref{typicalcor} and Lemmas~\ref{lowerdiss}-\ref{lowerinert} will be word-to-word the same as in 1d, the only difference being that we average over $\tilde{\ix}_1$ and not only in time and in probability. We will refer to this argument as the \enquote{1d restriction argument}.
\\ \indent
 In the following definition, (\ref{condi}-\ref{condii}) contain lower and upper estimates, while (\ref{condiii}) only contains an upper estimate.  The inequality $|\tilde{\psi}_1(s)|_{\infty} \leq  \max \tilde{\psi}_1(s)$ in (\ref{condi}) always holds, since the restriction $\tilde{\psi}_1(s)$ has zero mean and the length of $\T^1$ is $1$.

\begin{defi}
For a given solution $\ue(s)=\ue^{\omega}(s)$ and $K>1$, we denote by $L_K$ the set of all $(s,\tilde{\ix}_1,\omega) \in [t,t+T_0] \times \T^{d-1} \times \Omega$ such that the corresponding restrictions $\tilde{\psi}_1$ satisfy
\begin{align} \label{condi}
&K^{-1} \leq |\tilde{\psi}_1(s)|_{\infty} \leq \max \tilde{\psi}_{11}(s) \leq K
\\ \label{condii}
&K^{-1} \nu^{-1} \leq |\tilde{\psi}_1(s)|_{1,\infty} \leq K \nu^{-1}
\\ \label{condiii}
&|\tilde{\psi}_1(s)|_{2,\infty} \leq K \nu^{-2}.
\end{align}
\end{defi}

\begin{lemm} \label{typical}
There exist constants $C,K_1>0$ such that for $K \geq K_1$, $\rho(L_K) \geq C$. Here, $\rho$ denotes the product measure of the Lebesgue measures in time and space and $\Pe$ on $[t,t+T_0] \times \T^{d-1} \times \Omega$.
\end{lemm}
\indent
Let us denote by $O_K \subset [T_1,T_2]$ the set defined as $L_K$, but with relation (\ref{condii}) replaced by
\begin{equation} \label{condiibis}
K^{-1} \nu^{-1} \leq -\min_{\ix \in \T^d} \tilde{\psi}_{11}(s,\ix) \leq K \nu^{-1}.
\end{equation}

\begin{cor} \label{typicalcor}
For $K \geq K_1$ and $\nu < K_1^{-2}$, we have $\rho(O_K) \geq C$. Here, $C$ and $K_1$ are the same as in the formulation of Lemma~\ref{typical}.
\end{cor}

Now we fix
\begin{equation} \label{K}
K=K_1,
\end{equation}
and we choose
\begin{equation} \label{nu0eq}
\nu_0=\frac{1}{6} K^{-2};\ C_1=\frac{1}{4}K^{-2};\ C_2=\frac{1}{20}K^{-4}.
\end{equation}
In particular, we have $0<C_1 \nu_0 <C_2<1$: thus, the intervals $J_i$ are non-empty and non-intersecting for all $\nu \in (0,\nu_0]$.

The following results are also proved using the 1d restriction argument. Note that in the corresponding proofs in 1d, we use an argument in \cite{AFLV92} (cf. \cite{BorD}).

\begin{lemm} \label{lowerdiss}
For $\alpha \geq 0$ and $\ell \in J_1$,
$$
S^{\|}_{p,\alpha}(\ell \e_1) \overset{p,\alpha}{\gtrsim} \left\lbrace \begin{aligned} &\ell^{\alpha p},\ 0 \leq p \leq 1. \\ & \ell^{\alpha p} \nu^{-\alpha(p-1)},\ p \geq 1. \end{aligned} \right.
$$
\end{lemm}

\begin{lemm} \label{lowerinert}
For $\alpha \geq 0$ and $\ell \in J_2$,
$$
S^{\|}_{p,\alpha}(\ell \e_1) \overset{p,\alpha}{\gtrsim} \left\lbrace \begin{aligned} & \ell^{\alpha p},\ 0 \leq p \leq 1. \\ & \ell^{\alpha},\ p \geq 1. \end{aligned} \right.
$$
\end{lemm}

We will need the following Young-type inequality. It seems to be a well-known fact, at least in the case $p=2$. However, we were unable to find its proof in the literature.

\begin{lemm} \label{Young}
For all $p,\delta>0$ there exists a constant $K(p,\delta)$ such that we have
\begin{align} \label{Youngineq}
& |A+B|^p \leq (1+\delta) |A|^p+ K |B|^p,\ A,B \in \R^d.
\end{align}
\end{lemm}

\begin{proof} Without loss of generality, we can assume that $A$ and $B$ are colinear vectors, and thus reduce ourselves to the case where $A$ and $B$ are positive scalars.
\\ \indent
\textbf{Case $0 < p \leq 1$.} In this case, inequality (\ref{Youngineq}) holds with $K=1$ for all $\delta$, since by Minkowski's inequality applied to $A^{p}$ and $B^{p}$ we have
$$
((A^{p})^{1/p}+(B^{p})^{1/p})^{p} \leq A^p+B^p.
$$
\indent
\textbf{Case $p > 1$.}  We consider
$$
C=\zeta^{-1} A;\  D=(1-\zeta)^{-1} B,
$$
where by definition
$$
\zeta=(1+\delta)^{-1/(p-1)}.
$$
Since
$$
A+B=\zeta C+(1-\zeta) D
$$
and the function $x \mapsto x^p$ is convex, we get
\begin{align} \nonumber
(A+B)^p & \leq \zeta C^p+(1-\zeta) D^p
\\ \nonumber
& \leq \zeta^{-(p-1)} A^p+ (1-\zeta)^{-(p-1)} B^p
\\ \nonumber
& \leq (1+\delta) A^p+ (1-\zeta)^{-(p-1)} B^p,
\end{align}
which proves the lemma with $K(p,\delta)=(1-\zeta)^{-(p-1)}$.
\end{proof}

\begin{cor} \label{Youngcor}
For all $\alpha,p,\epsilon>0$ there exists a constant $K(p,\alpha,\epsilon)$ such that we have
\begin{align} \nonumber
& S_{p,\alpha}(\ve+\tilde{\ve}) \leq (1+\epsilon) S_{p,\alpha}(\ve)+ K S_{p,\alpha}(\tilde{\ve}),\ \ve \in \R^d,\ \tilde{\ve} \in \R^d,
\end{align}
\end{cor}

\begin{proof} For every $\ix$, Lemma~\ref{Young} yields that
\begin{align} \nonumber
&|u(\ix+\ve+\tilde{\ve})-u(\ix)|^p 
\\ \nonumber
& \leq (1+\epsilon)^{1/2 \alpha} |u(\ix+\ve)-u(\ix)|^p
\\ \nonumber
& +C(p,\alpha,\epsilon)  |u(\ix+\ve+\tilde{\ve})-u(\ix+\ve)|^p.
\end{align}
After averaging in $\ix$ we get
\begin{align} \nonumber
& \int_{\ix \in \T^d}{|u(\ix+\ve+\tilde{\ve})-u(\ix)|^p  d \ix}
\\ \nonumber
& \leq (1+\epsilon)^{1/2 \alpha} \int_{\ix \in \T^d}{|u(\ix+\ve)-u(\ix)|^p d \ix}
\\ \nonumber
& + C(p,\alpha,\epsilon)\ \int_{\ix \in \T^d}{|u(\ix+\tilde{\ve})-u(\ix)|^p d \ix}.
\end{align}
Applying again Lemma~\ref{Young}, we obtain that
\begin{align} \nonumber
& \Big( \int_{\ix \in \T^d}{|u(\ix+\ve+\tilde{\ve})-u(\ix)|^p  d \ix} \Big)^{\alpha}
\\ \nonumber
& \leq (1+\epsilon)^{1/2} \Big((1+\epsilon)^{1/2 \alpha} \int_{\ix \in \T^d}{|u(\ix+\ve)-u(\ix)|^p d \ix}  \Big)^{\alpha}
\\ \nonumber
& +C(\alpha,\epsilon)  \Big( C(p,\alpha,\epsilon) \int_{\ix \in \T^d}{|u(\ix+\tilde{\ve})-u(\ix)|^p d \ix}  \Big)^{\alpha}.
\end{align}
To prove the lemma's statement, it remains to take the expected value and to average in time in the inequality above.
\end{proof}
\indent
By Corollary~\ref{Youngcor}, for $\ell \in [0,1]$, $\tilde{\ve} \in \R^d$ we have
\begin{align} \label{smallangle}
& |S_{p,\alpha}(\ell \e_1+\tilde{\ve})-S_{p,\alpha}(\ell \e_1)| \leq \epsilon S_{p,\alpha}(\ell \e_1)+ K(p,\alpha,\epsilon) S_{p,\alpha}(\tilde{\ve}).
\end{align}
Now consider $\er \in \R^d,\ |\er|=\ell$ and denote by $\theta$ the angle between $\er$ and $\e_1$. Using (\ref{smallangle}), as well as respectively Lemma~\ref{upperdiss} and Lemma~\ref{lowerdiss} for (\ref{lowercordiss}) and Lemma~\ref{upperinert} and Lemma~\ref{lowerinert} for (\ref{lowerinertdiss}), we get the following result.

\begin{cor} \label{lowercor}
There exists a constant $\tilde{C}$ such that for $\alpha \geq 0$ and $|\er|=\ell \in J_1$, $|\theta| \leq \tilde{C}$,
\begin{align} \label{lowercordiss}
S_{p,\alpha}(\er) \overset{p,\alpha}{\gtrsim} \left\lbrace \begin{aligned} &\ell^{\alpha p},\ 0 \leq p \leq 1. \\ & \ell^{\alpha p} \nu^{-\alpha(p-1)},\ p \geq 1, \end{aligned} \right.
\end{align}
and for $\alpha \geq 0$ and $|\er|=\ell \in J_2$, $|\theta| \leq \tilde{C}$,
\begin{align} \label{lowerinertdiss}
S_{p,\alpha}(\er) \overset{p,\alpha}{\gtrsim} \left\lbrace \begin{aligned} & \ell^{\alpha p},\ 0 \leq p \leq 1. \\ & \ell^{\alpha},\ p \geq 1. \end{aligned} \right.
\end{align}
\end{cor}

Summing up the results above, averaging in $\er$ and using the definition of $S_{p,\alpha}(\ell)$, we obtain the following theorem.

\begin{theo} \label{avoir2}
For $\alpha \geq 0$ and $\ell \in J_1$,
$$
S_{p,\alpha}(\ell) \overset{p,\alpha}{\sim} \left\lbrace \begin{aligned} & \ell^{\alpha p},\ 0 \leq p \leq 1. \\ & \ell^{\alpha p} \nu^{-\alpha (p-1)},\ p \geq 1. \end{aligned} \right.
$$
On the other hand, for $\alpha \geq 0$ and $\ell \in J_2$,
$$
S_{p,\alpha} (\ell) \overset{p,\alpha}{\sim} \left\lbrace \begin{aligned} & \ell^{\alpha p},\ 0 \leq p \leq 1. \\ & \ell^{\alpha},\ p \geq 1. \end{aligned} \right.
$$
\end{theo}

The following result follows immediately from the definition (\ref{flatness}).

\begin{cor} \label{flatnesscor}
For $\ell \in J_2$, the flatness satisfies $F(\ell) \sim \ell^{-1}$.
\end{cor}

\subsection{Results in Fourier space} \label{four}

By (\ref{integers}), for all $m \geq 1$ we have
\begin{equation} \label{superalg}
\lbrace |\hat{\ue}(\en)|^2 \rbrace \leq (2 \pi |\en|)^{-2m} \lbrace \Vert \ue \Vert_m^2 \rbrace \overset{m}{\sim} (\nu |\en|)^{-2m} \nu.
\end{equation}
Thus, for $|\en| \succeq \nu^{-1}$, $\lbrace|\hat{\ue}(\en)|^2\rbrace$ decreases super-algebraically.
\smallskip
\\ \indent
To estimate the $H^s$ norms of $\ue$ for $s \in (0,1)$, we proceed in the same way as in \cite{BorW}. Namely, we use the formula (\ref{Sobolevfrac}) and Theorem~\ref{avoir2} to transform information about increments into information about Sobolev norms. For the sake of completeness, we give here the proof in the case $s=1/2$.

\begin{lemm} \label{H01}
For $s \in (0,1/2)$,
$$
\lbrace ( \Vert \ue \Vert^{'}_s)^2 \rbrace \overset{s}{\sim} 1.
$$
On the other hand,
$$
\lbrace ( \Vert \ue \Vert_{1/2}^{'})^2 \rbrace \sim |\log \nu|.
$$
Finally, for $s \in (1/2,1)$,
$$
\lbrace ( \Vert \ue \Vert_s^{'})^2 \rbrace \overset{s}{\sim} \nu^{-(2s-1)}.
$$
\end{lemm}

\begin{proof}
By (\ref{Sobolevfrac}) we have
\begin{align} \nonumber
\left\|\ue\right\|'_{1/2} &\sim \Bigg(\int_{\ix \in \T^d,\ |\er| \leq 1}{\frac{|\ue(\ix+\er)-\ue(\ix)|^2}{|\er|^{1+d}}} d \ix d \er \Bigg)^{1/2}
\\ \nonumber
& \sim \Bigg( \int_0^1 \frac{1}{\ell^{2}} \Big(\int_{\ix \in \T^d,\ \y \in S^{d-1}} {|\ue(\ix+\ell \y)-\ue(\ix)|^2 d \ix d \y} \Big) d \ell \Bigg)^{1/2}.
\end{align}
Consequently, by Fubini's theorem, we get
\begin{align} \nonumber
(\lbrace \left\|\ue\right\|^{'}_{1/2})^2 \rbrace &\sim \int_{0}^{1}{\frac{S_2(\ell)}{\ell^2} d\ell}
=\int_{J_1}{\frac{S_2(\ell)}{\ell^2} d\ell}+\int_{J_2}{\frac{S_2(\ell)}{\ell^2} d\ell}+\int_{J_3}{\frac{S_2(\ell)}{\ell^2} d\ell}.
\end{align}
By Theorem~\ref{avoir2} we get
$$
\int_{J_1}{\frac{S_2(\ell)}{\ell^2} d\ell} \sim \int_{0}^{C_1 \nu}{\frac{\ell^2 \nu^{-1}}{\ell^2} d\ell} \sim 1
$$
and
$$
\int_{J_2}{\frac{S_2(\ell)}{\ell^2} d\ell} \sim \int_{C_1 \nu}^{C_2}{\frac{\ell}{\ell^2} d\ell} \sim |\log  \nu|,
$$
respectively. Finally, by Lemma~\ref{upperinert} we get
$$
\int_{J_3}{\frac{S_2(\ell)}{\ell^2} d\ell} \leq C C_2^{-2} \leq C.
$$
Thus,
$$
(\lbrace \left\|\ue\right\|^{'}_{1/2})^2 \rbrace \sim |\log \nu|.
$$
\end{proof}
\indent
The results above and the relation (\ref{integers}) tell us that $\lbrace|\hat{\ue}(\en)|^2\rbrace$ decreases very fast for $| \en | \gtrsim \nu^{-1}$ and that for $s \geq 0$ the sums
$$
\sum{| \en |^{2s} \lbrace |\hat{\ue}(\en)|^2}\rbrace
$$
have exactly the same behaviour as the partial sums
$$
\sum_{| \en | \leq \nu^{-1}}{| \en |^{2s+(1-d)} | \en |^{-2}}
$$
in the limit $\nu \rightarrow 0^+$. Therefore we can conjecture that for $| \en | \lesssim \nu^{-1}$, we have  $\sum_{|\en| \sim k}{\lbrace|\hat{\ue}(\en)|^2} \rbrace \sim k^{-2}$.
\\ \indent
A result of this type actually holds for the layer-averaged Fourier coefficients as long as $| \en |$ is not too small, i.e. in the inertial range $J_2$. The proof is a little bit more delicate than in 1d, since the upper estimate does not follow directly from the bound in $W^{1,1}$. We use a version of the Wiener-Khinchin theorem, which states that for any function $\ve \in L_2$ and any $\y \in \R^d$, one has
\begin{equation} \label{WK}
|\ve(\cdot+\y)-\ve(\cdot)|^2=4\sum_{\en \in \Z^d}{ \sin^2 (\pi \en \cdot \y) |\hat{\ve}(\en)|^2}.
\end{equation}
%



\begin{theo} \label{spectrinert}
There exists $M \geq 2$ such that for $k^{-1} \in J_2$, we have $E(k) \sim k^{-2}$.
\end{theo}

\begin{proof}
We recall that by the definition (\ref{spectrum}),
$$
E(k) = k^{-1} \sum_{| \en | \in [M^{-1}k,Mk]}{\Big\{  {|\hat{\ue}(\en)|^2} \Big\} }.
$$
Thus, it suffices to prove that
\begin{equation} \label{spectrinertequiv}
E'(k)=\sum_{| \en | \in [M^{-1}k,Mk]}{ |\en|^2 \lbrace |\hat{\ue}(\en)|^2 \rbrace } \sim k.
\end{equation}
In the following, dependence on $M$ will always be explicit. We begin by proving the upper estimate. First, we note that we have
$$
|\ve|^2 \sim \int_{\y \in \mathcal{S}^{d-1}} \sin^2(\pi \ve \cdot \y),
$$
uniformly for $|\ve| \leq 1/2$. Thus, we get
\begin{align} \nonumber
& \sum_{| \en | \in [k/2,2k]}{ \lbrace |\hat{\ue}(\en)|^2 \rbrace }
\\ \nonumber
&\sim k^{d-1} \int_{\y \in k^{-1} \mathcal{S}^{d-1}} \sum_{| \en | \in [k/2,2k]}
{ \sin^2(\pi \en \cdot \y/4) \lbrace |\hat{\ue}(\en)|^2 \rbrace d \y}
\\ \nonumber
& \lesssim k^{d-1} \int_{\y \in k^{-1} \mathcal{S}^{d-1}} \sum_{\en \in \Z^d}{ \sin^2(\pi \en \cdot \y/4) \lbrace |\hat{\ue}(\en)|^2 \rbrace d \y}.
\end{align}
Then by (\ref{WK}) and Lemma~\ref{upperinert} we get
\begin{align} \nonumber
\sum_{| \en | \in [k/2,2k]}{ \lbrace |\hat{\ue}(\en)|^2 \rbrace } & \lesssim S_2(k^{-1}/4) \lesssim k^{-1}.
\end{align}
Consequently, we obtain the upper bound
$$
E'(k) \lesssim  M k.
$$
On the other hand, we get
\begin{equation} \label{spectrinertupper1}
\sum_{ |\en | < M^{-1} k}{ |\en|^2 \lbrace |\hat{\ue}(\en)|^2 \rbrace } \leq C M^{-1} k
\end{equation}
and (summing over layers of the form $[M^{2N-1} k,\ M^{2N+1} k]$):
\begin{equation} \label{spectrinertupper2}
\sum_{|\en| > M k}{\lbrace |\hat{\ue}(\en)|^2 \rbrace } \leq C M^{-1} k^{-1}.
\end{equation}
The lower bound for $E'(k)$ is then obtained in exactly the same way as in 1d. Namely, we note that for $\y  \in k^{-1} S^{d-1}$ and $\en \in \Z^d$, we have
$$
|\en|^2 \geq k^2 \pi^{-2} \sin^2 (\pi \en \cdot \y).
$$
Consequently, 
\begin{align} \nonumber
&\sum_{|\en| \leq M k}{ |\en|^2 \lbrace |\hat{\ue}(\en)|^2 \rbrace } 
\\ \nonumber
& \geq  c_d^{-1} k^{d-1} \int_{\y \in k^{-1} S^{d-1}}\sum_{|\en| \leq M k}{  k^2 \pi^{-2}  \sin^2 (\pi \en \cdot \y) \lbrace |\hat{\ue}(\en)|^2 \rbrace } d \y
\\ \nonumber
&\geq  k^2 \pi^{-2} \Big( c_d^{-1} k^{d-1} \int_{\y \in k^{-1} S^{d-1}}{\sum_{\en \in \Z^d}{ \sin^2 (\pi \en \cdot \y) \lbrace |\hat{\ue}(\en)|^2 \rbrace } d \y}
\\ \nonumber
&- \sum_{|\en| > M k}{\lbrace |\hat{\ue}(\en)|^2 \rbrace } \Big),
\end{align}
where $c_d$ is the surface of $\mathcal{S}^{d-1}$. Using (\ref{WK}), (\ref{spectrinertupper2}) and the definition of $S_2$ we get
\begin{align} \nonumber
\sum_{|\en| \leq M k}{ | \en|^2 \lbrace |\hat{\ue}(\en)|^2 \rbrace } &\geq  k^2 \pi^{-2} (S_2(k^{-1})/4-C M^{-1} k^{-1}).
\end{align}
Finally, Theorem~\ref{avoir2} yields that
\begin{equation} \nonumber
\sum_{|\en| \leq M k}{ |\en|^2 \lbrace |\hat{\ue}(\en)|^2 \rbrace }  \geq (C-C M^{-1}) k.
\end{equation}
Now we use (\ref{spectrinertupper1}) and we choose $M \geq 1$ large enough to obtain (\ref{spectrinertequiv}).
\end{proof}

\begin{rmq} \label{spectrinertrmq}
We actually have
$$
\Bigg\{ \Bigg( k^{-1} \sum_{|\en| \in [M^{-1}k,Mk]}{|\hat{\ue}(\en)|^2} \Bigg)^{\alpha} \Bigg\} \overset{\alpha}{\sim} k^{-2\alpha},\quad \alpha>0.
$$
The upper bound is proved in the same way as previously and then the lower bound follows from H{\"o}lder's inequality and the lower bound in Theorem \ref{spectrinert}.
\end{rmq}

\section{Stationary measure and related issues} \label{stat}

Here we very briefly discuss the stationary measure for the equation (\ref{whiteBurgers}). The scheme of the proofs is similar to the one in the 1d setting \cite{BorW}, and therefore we will not give the details. The only major difference is that the contraction argument for $u$ in $L_1$ should be replaced by a contraction argument for the potential $\psi$ in $L_{\infty}$.
\\ \indent
We begin by studying the equation (\ref{HJint}). Its solutions $\psi$ form a Markov process: this is proved using a simplified version of the coupling argument for the 2D Navier-Stokes equations \cite{KuSh12}.
\\ \indent
For a given value of $\omega \in \Omega$, we denote by $S_t^{\omega}$ the semigroup acting on $H^{s_0+1}$ (see (\ref{s0}) for the definition of $s_0$) defined by
$$
\psi^0 \mapsto \psi(t).
$$
Now consider the dual semigroup $S_t^{*}$ acting on the space of probability measures on $H^{s_0+1}$. A \textit{stationary measure} is a probability measure on $H^{s_0+1}$ invariant by $S_t^{*}$ for every $t$. A \textit{stationary solution} is a solution $\psi(t,x)$ of (\ref{HJint}) such that the law of $\psi(t)$ does not depend on $t$ for $t \geq 0$ and thus is a stationary measure for (\ref{HJint}). 
\\ \indent
Existence of a stationary measure for (\ref{HJint}) follows from the estimates in Section~\ref{upper} by the Bogolyubov-Krylov argument. Now we define the Lipschitz-dual metric with respect to $L_p,\ 1 \leq p \leq \infty$.

\begin{defi}
For a continuous real-valued function $g$ on $L_p,\\ 1 \leq p \leq \infty$, we define its Lipschitz norm as
$$
|g|_{L(p)}:=\sup_{L_p}{|g|}+|g|_{Lip},
$$
where
$$
\sup_{L_p}{|g|}=\sup_{\ix \in L_p}{|g(\ix)|}
$$
and $|g|_{Lip}$ is the Lipschitz constant of $g$, i.e.
$$
|g|_{Lip}=\sup_{\ix,\y \in L_p,\ \ix \neq \y}{\frac{|g(\ix)-g(\y)|}{|\ix-\y|_{p}}}.
$$
The set of continous functions with finite Lipschitz norm will be denoted by $L(p)=L(L_p)$.
\end{defi}

\begin{defi}
For two Borel probability measures $\mu_1,\mu_2$ on $L_p,\\ 1 \leq p \leq \infty$, we denote by $\Vert \mu_1-\mu_2 \Vert^*_{L(p)}$ the Lipschitz-dual distance
$$
\Vert \mu_1-\mu_2 \Vert^*_{L(p)}:=\sup_{g \in L(p),\ |g|_{L(p)} \leq 1}{\Big| \int_{S^1}{g(v) \mu_1(dv)}-\int_{S^1}{g(v) \mu_2(dv)} \Big|}.
$$
\end{defi}

Now we prove a standard contraction property for $\psi$ in $L_{\infty}$. This property can be proved using a Lagrangian formulation for the solution to (\ref{HJint}) (see for instance \cite{GIKP05}); here we give a more elementary proof.

\begin{lemm} \label{contract}
Let us take two different $C^{\infty}$-smooth initial conditions $\psi^0_1$ and $\psi^0_2$. Consider a fixed $\omega \in \Omega$. We have
$$
|S_t^{\omega} \psi^0_1-S_t^{\omega}\psi^0_2|_{\infty} \leq |\psi^0_1-\psi^0_2|_{\infty},\ t \geq 0.
$$
\end{lemm}

\begin{proof}
Denote by $\phi$ the difference $S_t^{\omega} \psi^0_2-S_t^{\omega} \psi^0_1$. Substracting the equation satisfied by $S_t^{\omega} \psi^0_2$ from the one satisfied by $S_t^{\omega} \psi^0_1$, we get
\begin{align} \nonumber
\phi_t &= (S_t^{\omega} \psi^0_2-S_t^{\omega} \psi^0_1)_t
\\ \nonumber
&= -\Big(f(\nabla (S_t^{\omega} \psi^0_2))-f(\nabla (S_t^{\omega} \psi^0_1)) \Big)+\nu \Delta(S_t^{\omega} \psi^0_2-S_t^{\omega} \psi^0_1).
\end{align}
Now consider the function $\be(t,x)$ such that for $1 \leq i \leq d$,  its $i$-th component is given by
$$
\frac{(f(\nabla (S_t^{\omega} \psi^0_2))-f(\nabla( S_t^{\omega} \psi^0_1))) \phi_i}{|\nabla \phi|^2} .
$$
Note that this function is well-defined at points where $\nabla \phi=0$ since by definition $\nabla \phi=\nabla(S_t^{\omega} \psi^0_2)-\nabla(S_t^{\omega} \psi^0_1)$, and $f$ is $C^{\infty}$-smooth.
\\ \indent
We see that $\phi$ satisfies the \textit{linear} parabolic equation
$$
\phi_t=-(\be(t,x) \cdot \nabla) \phi+\nu \Delta \phi.
$$
Consequently, by the maximum principle \cite{Lan98} we get the lemma's statement.
\end{proof}
\indent
Since $C^{\infty}$ is dense in $L_{\infty}$, we can extend the notion of solutions to (\ref{HJint}) to solutions with initial conditions in $L_{\infty}$ . The definitions of $S_t$ and $S_t^{*}$ can be extended accordingly. Note that the parabolic smoothing effect due to the viscous term yields that these solutions instantaneously become smooth solutions to (\ref{HJint}).
\\ \indent
Now we use a coupling argument and a "small-noise zone" argument to prove the following crucial lemma. The proof is almost word-to-word the same as in 1d. The only difference is that now when the noise is small, the \textit{gradient} of the solution to (\ref{HJint}) is small. Therefore we consider the space $L(\infty)/\R$ of Lipschitz functions on the space $L_{\infty}/\R$ with a norm defined in the same way as the $L(\infty)$-norm. 

\begin{lemm} \label{algCVlemm}
There exist positive constants $C',\delta$ such that for $u^0_1, u^0_2 \in L_{\infty}$ we have
\begin{equation} \label{algCVlemmformula}
\Vert S_t^{*} \delta_{u^0_1}-S_t^{*} \delta_{u^0_2} \Vert^*_{L(\infty)/\R} \leq C't^{-\delta},\qquad t \geq 1.
\end{equation}
\end{lemm}

\indent
Now we look at the equation (\ref{whiteBurgers}). In the same way as above for (\ref{HJint}), we can define the semigroups $\tilde{S}_t^{\omega}$ and $\tilde{S}_t^{*}$, acting respectively on $L(1)$ and on the space of probability measures on $L(1)$. We consider two solutions $\psi_1,\psi_2$ to (\ref{HJint}) with the same noise and different initial conditions, as well as the corresponding solutions $\ue_1,\ue_2$ to (\ref{whiteBurgers}). By (GN) we get
\begin{align} \nonumber
|\ue_1-\ue_2|_1 & \lesssim |\psi_1-\psi_2-\int_{\T^d}{(\psi_1-\psi_2)}|_{1}  |\nabla(\psi_1-\psi_2)|_{1,1} 
\\ \nonumber
&\lesssim |\psi_1-\psi_2-\int_{\T^d}{(\psi_1-\psi_2)}|_{\infty}  |\ue_1-\ue_2|_{1,1}.
\end{align}
This inequality allows us to obtain the following result.

\begin{theo} \label{algCV}
There exist positive constants $C,\delta'$ such that we have
\begin{equation} \label{algCVformula}
\Vert \tilde{S}_t^{*}\mu_1-\tilde{S}_t^{*}\mu_2 \Vert^*_{L(1)} \leq C't^{-\delta'},\qquad t \geq 1,
\end{equation}
for any probability measures $\mu_1$, $\mu_2$ on $L(1)$.
\end{theo}

\indent
The estimates for Sobolev norms and small-scale quantities proved in the previous sections still hold for a stationary solution of (\ref{whiteBurgers}). Indeed, it suffices to consider a random initial condition $u^0$ with distribution $\mu$. It follows that those estimates still hold when averaging in time and in ensemble (denoted by $\lbrace \cdot \rbrace$) is replaced by averaging solely in ensemble, i.e. by integrating  with respect to $\mu$. Namely, we get the following results, which follow from Theorem~\ref{avoir}, Theorem~\ref{avoir2} and Remark~\ref{spectrinertrmq}, respectively.

\begin{theo} \label{avoirstat}
For $m=0$ and $p \in [1,\infty]$, $m=1$ and $p \in [1,\infty)$, or $m \geq 2$ and $p \in (1,\infty)$,
\begin{equation} \label{asympstat}
\Big( \int {\left|\ue(s)\right|_{m,p}^{\alpha} d \mu} \Big)^{1/\alpha}  \overset{m,p,\alpha}{\sim} \nu^{-\gamma},\quad \alpha>0.
\end{equation}
\end{theo}

\begin{theo} \label{avoir2stat}
For $\alpha \geq 0$ and $\ell \in J_1$,
$$
\int{S_{p,\alpha}(\ell) d \mu} \overset{p,\alpha}{\sim} \left\lbrace \begin{aligned} & \ell^{\alpha p},\ 0 \leq p \leq 1. \\ & \ell^{\alpha p} \nu^{-\alpha (p-1)},\ p \geq 1. \end{aligned} \right.
$$
On the other hand, for $\alpha \geq 0$ and $\ell \in J_2$,
$$
\int{S_{p,\alpha}(\ell) d \mu } \overset{p,\alpha}{\sim} \left\lbrace \begin{aligned} & \ell^{\alpha p},\ 0 \leq p \leq 1. \\ & \ell^{\alpha},\ p \geq 1. \end{aligned} \right.
$$
\end{theo}

\begin{theo} \label{spectrinertstat}
For $k$ such that $k^{-1} \in J_2$, we have 
$$
\displaystyle\int { \Bigg( k^{-1} \sum_{|\en| \in [M^{-1}k,Mk]}{|\hat{\ue}(\en)|^2}  \Bigg)^{\alpha} d \mu } \overset{\alpha}{\sim} k^{-2\alpha},\quad \alpha>0.
$$
\end{theo}

\section*{Appendix 1: well-posedness of (\ref{HJint})} \label{appwp}

In this appendix, we consider the well-posedness of the Cauchy problem given by  (\ref{HJint}), a.s. An analogous problem has been considered by Da Prato and Zabczyk in \cite[Chapter 14]{DZ96}; however, their results are weaker than ours since they consider a white noise which is not smooth in space.
\\ \indent
Here, the functions whose Sobolev norms we consider do not necessarily have zero mean value in space. The only thing that changes is that now in the expressions for the Sobolev norms $W^{m,p}$ (resp. $H^s$) we have to add the norm in $L_p$ (resp. $L_2$) to the formulas 
in Section~\ref{sob}. We use the standard notation $C(I,W^{m,p})$ for the space of continuous (in time) functions defined on the interval $I$ with values in $W^{m,p}$ with the corresponding supremum norm. The space $C(I,C^{\infty})$ will denote the intersection
$$
\cap_{m \geq 0}{C(I,H^m)}.
$$
\indent
We begin by considering mild solutions in $H^{s_0+1}$, in the spirit of \cite{DZ92,DZ96}. We recall that $s_0$ is the integer given by (\ref{s0}). Then, by a bootstrap argument, we prove that for strictly positive times these solutions are actually smooth. Then upper estimates (cf. Section~\ref{upper}) allow us to prove that such mild solutions are global.
\\ \indent
We recall that there exists an event $\Omega_1$ such that $\Pe(\Omega_1)=1$ and for $\omega \in \Omega_1$, the Wiener process $w(t)$ belongs to $C([0,+\infty),C^{\infty})$.
We also recall the notation $L=-\Delta$, and the fact that the initial condition $\psi_0$ and the function $f$ in the nonlinearity are $C^{\infty}$-smooth.
\\ \indent
By a scaling argument, we can restrict ourselves to the equation (\ref{whiteBurgers}) with $\nu=1$. We will denote by $S_L(t)$ the heat semigroup $e^{-tL}$. We recall that for $v \in L_2$ the function $S_L(t) v(\ix)$ is given by:
\begin{equation} \label{heat}
S_L(t) v(\ix)=\sum_{\ka \in \Z^d}{e^{-4 \pi^2 |\ka|^2 t} \hat{v}_{\ka} e^{2 \pi i \ka \cdot \ix}}.
\end{equation}
Finally, we denote by $w_L$ the stochastic convolution
$$
w_L(t)=\int_{0}^{t}{S_L(t-\tau)dw(\tau)}.
$$
For $\omega \in \Omega_2,\ \Pe(\Omega_2)=1$, this quantity belongs to $C([0,+\infty),C^{\infty})$. From now on, we suppose that $\omega$ belongs to $\Omega_1 \cap \Omega_2$.
\medskip
\\ \indent
Following Da Prato and Zabczyk \cite[Chapter 14]{DZ96}, we consider a mild form of (\ref{HJint}) for $Y(t)=\psi(t)-w_L(t)$:
\begin{equation} \label{mildHJ}
Y(t)=S_L(t) \psi_0+\int_{0}^{t}{S_L(t-\tau)(f(\nabla Y(\tau)+\nabla w_L(\tau))) d \tau}.
\end{equation}
The heat semigroup defines a contraction in each Sobolev space $H^s$. On the other hand, we have the following lemma.

\begin{lemm} \label{Lip}
The mapping
$$
Z \mapsto f(Z):\ H^{s_0} \rightarrow H^{s_0}
$$
is locally Lipschitz on bounded subsets of $H^{s_0}$.
\end{lemm}

\textbf{Proof:} 
It suffices to develop $(f(Z_1)-f(Z_2))^{(s_0)}$ using Leibniz's formula ($s_0$ being an integer) and then to use the Sobolev injection (\ref{Sobinj}).\ $\square$

\begin{lemm} \label{heatconvol}
For any $s \geq 0$, the operator
$$
Z \mapsto \Big( t \mapsto \int_{0}^{t}{S_L(t-\tau) Z(\tau) d \tau} \Big)
$$
maps bounded subsets of $C([0,T),H^{s})$ into bounded subsets of
\\
$C([0,T),H^{(s+3/2)})$.
\end{lemm}

\textbf{Proof:}
Fix $s \geq 0$. By (\ref{Sobolevspectr}) and (\ref{heat}), for $\tau \in [0,t)$ we have
\begin{align} \nonumber 
& \Vert S_L(t-\tau) Z(\tau) \Vert^2_{s+3/2}
\\ \nonumber
& \sim |(\hat{Z}(\tau))_\zero|^2+\sum_{\ka \in \Z^d}{|\ka|^{2s+3} e^{-4 \pi^2 |\ka|^2 (t-\tau)} |(\hat{Z}(\tau))_\ka|^2}
\\ \nonumber
& \lesssim |(\hat{Z}(\tau))_\zero|^2+\Big(\max_{\ka' \in \Z^d}{|\ka'|^{3} e^{-4 \pi^2 |\ka'|^2 (t-\tau)}} \Big) \sum_{\ka \in \Z^d}{|\ka|^{2s} |(\hat{Z}(\tau))_\ka|^2}
\\ \nonumber
& \lesssim \Big(1+\max_{\ka' \in \Z^d}{|\ka'|^{3} e^{-4 \pi^2 |\ka'|^2 (t-\tau)}} \Big) \Vert Z(\tau)\Vert^2_s.
\\ \nonumber 
& \lesssim C \Big[1+(t-\tau)^{-3/2} \Big] \Vert Z(\tau)\Vert^2_s.
\end{align}
To prove the lemma's statement, it remains to observe that
$$
\int_{0}^{t}{(1+(t-\tau)^{-3/2})^{1/2} d \tau}<+\infty.\ \square
$$
\medskip
\\ \indent
Lemma~\ref{Lip}, Lemma~\ref{heatconvol} for $s=s_0$ and the Cauchy-Lipschitz theorem imply that the equation (\ref{mildHJ}) has a unique local solution in $H^{s_0+1}$.
\\ \indent
Now consider such a solution $Y$. We want to prove that this solution belongs to $C^{\infty}$ for all $t>0$. For this, it suffices to prove that for $s \geq s_0+1$, a solution $Y \in H^s$ lies in the space $H^{(s+1/2)}$. We will need the following result:

\begin{lemm} \label{Nemit}
For $s \geq s_0$, the mapping
$$
Z \mapsto f(Z):\ H^{s} \rightarrow H^{s}
$$
is bounded on bounded subsets of $H^{s}$.
\end{lemm}

\textbf{Proof:} 
An analogous lemma is proved in a more general setting for Sobolev spaces on $\R^n$ in \cite{BBM00}. We use some arguments from this paper.
\\ \indent
For the case when $s$ is integer, we proceed in the same way as in the proof of Lemma~\ref{Lip}, using Leibniz's formula and then (\ref{Sobinj}).
\\ \indent
Now consider the case when $s$ is non-integer. For simplicity, we will only consider the case $s_0 < s < s_0+1$; the general case follows from Leibniz's formula. Denote by $\tilde{s}$ the quantity $s-s_0$.
\\ \indent
Consider $Z$ such that $\Vert Z \Vert_{s} \leq N$. In this case, by the definition (\ref{Sobolevfrac}) we have:
\begin{align} \nonumber
& \Vert f(Z) \Vert^2_s \sim |f(Z)|^2+ 
\\  \label{f(Z)} 
& \int_{\ix \in \T^{d},\ |\er| \leq 1}{\frac{|(f(Z))^{(s_0)}(\ix+\er)-(f(Z))^{(s_0)}(\ix)|^2}{|\er|^{2 \tilde{s}+d}}}\ d \ix\ d \er
\end{align}
The least regular term in Leibniz's formula for $(f(Z))^{(s_0)}$ corresponds to
\\
$f'(Z)Z^{(s_0)}$. Therefore it suffices to bound the corresponding term in (\ref{f(Z)}) by $C(N)$.
\begin{align} \nonumber
& \int_{\ix \in \T^{d},\ |\er| \leq 1}  {  \frac{|f'(Z)(\ix+\er)Z^{(s_0)}(\ix+\er)-f'(Z)(\ix)Z^{(s_0)}(\ix)|^2}{|\er|^{2 \tilde{s}+d}}}\ d \ix\ d \er
\\ \nonumber
\lesssim  & \int_{\ix \in \T^{d},\ |\er| \leq 1}{\frac{|f'(Z)(\ix+\er)Z^{(s_0)}(\ix+\er)-f'(Z)(\ix+\er)Z^{(s_0)}(\ix)|^2}{|\er|^{2 \tilde{s}+d}}}\ d \ix\ d \er
\\ \nonumber
& + \int_{\ix \in \T^{d},\ |\er| \leq 1}{\frac{|f'(Z)(\ix+\er)Z^{(s_0)}(\ix)-f'(Z)(\ix)Z^{(s_0)}(\ix)|^2}{|\er|^{2 \tilde{s}+d}}}\ d \ix\ d \er
\\ \nonumber
\lesssim &\ |f'(Z)|^2_{\infty} \int_{\ix \in \T^{d},\ |\er| \leq 1}{\frac{|Z^{(s_0)}(\ix+\er)-Z^{(s_0)}(\ix)|^2}{|\er|^{2 \tilde{s}+d}}}\ d \ix\ d \er
\\ \label{reg}
&  + \int_{\ix \in \T^{d},\ |\er| \leq 1}{\frac{|f'(Z)(\ix+\er)-f'(Z)(\ix)|^2}{|\er|^{2 \tilde{s}+d}} |Z^{(s_0)}(\ix)|^2}\ d \ix\ d \er
\\ \nonumber
\lesssim & C(|Z|_{\infty}) \Vert Z \Vert_s^2+ \int_{\ix \in \T^{d},\ |\er| \leq 1}{|\er|^{2-2 \tilde{s}-d} |Z^{(s_0)}(\ix)|^2}\ d \ix\ d \er
\\ \nonumber
\lesssim & C(N)+C(N) \Vert Z \Vert_{s_0}^2 \lesssim C(N).
\end{align}
Indeed, the rest of the right-hand side in (\ref{f(Z)}) is more regular and can be bounded by $C(N)$ in the same way as the term in (\ref{reg}).
\ $\square$

\begin{theo} \label{bootstrap}
Consider a local solution $Y$ of (\ref{mildHJ}) in $H^{s_0+1}$ defined on an interval $[0,T)$. If for some $s \geq s_0+1$, $Y$ belongs to $C([0,T),H^s)$, then $Y$ actually belongs to $C([0,T),H^{(s+1/2)})$.
\end{theo}

\textbf{Proof:} By Lemma~\ref{Nemit} we have
$$
\nabla (f(Y(\tau)+w_L(\tau))) \in C([0,T),H^{s-1}),
$$
and thus by Lemma~\ref{heatconvol} we get
$$
\int_{0}^{t}{S_L(t-\tau) \nabla(f(Y(\tau)+w_L(\tau))) d\tau} \in C([0,T),H^{(s+1/2)}).
$$
Since $Y$ is a solution of (\ref{mildHJ}) and the semigroup $S_L$ is smoothing,
$$
Y(t)=S_L(t) \psi_0+\int_{0}^{t}{S_L(t-\tau) \nabla(f(Y(\tau)+w_L(\tau))) d \tau}
$$
belongs to the space $C([0,T),H^{s+1/2}).$\ $\square$
\medskip
\\ \indent
Thus, we have proved existence and uniqueness of a local solution to (\ref{whiteBurgers}), which is $C^{\infty}$-smooth in space for $t>0$. To see that this solution is necessarily global, it suffices to observe that for any $\tau,\tau'>0$ it satisfies estimates which hold uniformly in time for $t \in [\tau,\tau+\tau']$: see Remark~\ref{upperrmq}.

\section*{Appendix 2: proof of Lemma~\ref{linalg}} \label{appalg}

We recall the statement of the lemma.
\medskip
\\ \indent
\textit{For every $m,d \geq 1$, there exists a finite set $\Pi_m^d$ of homogeneous polynomials of degree 1 in $d$ variables $X_1,\dots,X_d$ with integer coefficients, such that their $m$-th powers form a basis for the vector space of homogeneous polynomials of degree $m$ in $d$ variables.}
\medskip
\\ \indent

\begin{proof} The case $d=1$ is trivial. In the case $d=2$, we consider the matrix of the $m$th powers of $X_1,X_1+X_2,\dots,X_1+mX_2$ written in the canonical basis $(X_1^m,X_1^{m-1}X_2,\dots, X_1 X_2^{m-1},X_2^m)$. We get
\begin{displaymath}
\left( \begin{array}{ccccc}
1 & 0 & 0 & 0 & \ldots \\ \\
1 & m & \binom{m}{2} & \binom{m}{3} & \ldots \\ \\
1 & 2m & 2^2 \binom{m}{2} & 2^3 \binom{m}{3} & \ldots \\ \\
\vdots & \vdots & \vdots & \vdots & \ddots
\end{array} \right)
\end{displaymath}
Dividing the $n$-th column by $\binom{m}{n}$ for every $n$, we obtain Vandermonde's matrix $V(0,\dots,m)$, which is invertible \cite[Chap. 7, \S 3, ex.5]{Lan72}. Thus, we may choose
$$
\Pi_m^2= \Big\{ X_1,\ X_1+X_2,\dots, X_1+mX_2 \Big\}.
$$
Finally, the case $d \geq 3$ follows by induction on $d$. Indeed, all monomials of degree $m$ can be written as
$$
X_1^{m-n} P_{n}(X_2,\dots,X_d),\quad 0 \leq n \leq m,
$$
where $\deg P_n=n$. To deal with the case $n=0$, it suffices to add $X_1$ to the set $\Pi_m^d$. For $n \geq 1$, the statement for $d-1$ tells us that $P_{n}(X_2,\dots,X_{d})$ can be written as a finite linear combination 
$$
\sum_{i=1}^{I(n)}{(L_i(X_2,\dots,X_{d}))^{n}},
$$
where the $L_i$ are homogeneous polynomials of degree $1$.
\\ \indent
Now, for every $n,\ 1 \leq n \leq m$ and every $i,\ 1 \leq i \leq I(n)$, consider $L_i$ and $X_1$ as the new independent variables, apply the lemma's statement for $d=2$ and add the resulting polynomials to the set $\Pi_m^d$. At the end of the procedure we get a generating family, and then a basis for $\Pi_m^d$.
\end{proof}

\section*{Acknowledgements}
\indent
I am very grateful to A. Biryuk, A. Debussche, K. Khanin, S. Kuksin,
\\
A. Shirikyan and J. Vovelle for helpful discussions and to B. Da Costa for the proof of Lemma~\ref{linalg}. A part of the present work was done during my stays at Laboratoire AGM, University of Cergy-Pontoise and D{\'e}partement de Physique Th{\'e}orique, University of Geneva, supported by the grants ERC BLOWDISOL and ERC BRIDGES: I would like to thank all the faculty and staff, and especially the principal investigators F.Merle and J.-P.Eckmann, for their hospitality. Moreover, I would like to thank J.-P.Eckmann for his remarks on a draft of the paper. Finally, I would like to thank the anonymous referee for his careful reading and in particular for spotting a mistake in one of the proofs.

\bibliographystyle{plain}
\bibliography{Bibliogeneral}

\end{document}